\documentclass[preprint]{imsart}
\usepackage{amsfonts, amsmath}
\usepackage{amsthm}
\usepackage{thmtools,thm-restate}
\usepackage{fullpage}
\usepackage{amssymb}
\usepackage{natbib}
\usepackage[dvipsnames]{xcolor}
\usepackage[colorlinks=true, linkcolor=red, urlcolor=blue, citecolor=blue]{hyperref}
\usepackage{xr}

\newtheorem*{theorem*}{Theorem}
\newtheorem*{lemma*}{Lemma}

\newcommand{\Bin}{\mathrm{Bin}}

\usepackage{fullpage}
\usepackage{natbib}
\usepackage{amsthm}
\usepackage{color}
\usepackage{bbm}
\theoremstyle{plain}

\newtheorem{lemma}{Lemma}
\newtheorem{definition}{Definition}

\newtheorem{remark}{Remark}

\newtheorem{claim}{Claim}

\newcommand{\reals}{\mathbb{R}}
\newcommand{\E}{\mathbb{E}}

\newcommand{\argmin}[1]{\underset{#1}{\mathrm{argmin}}}
\newcommand{\argmax}[1]{\underset{#1}{\mathrm{argmax}}}

\newcommand{\PP}{\mathbb{P}}
\newcommand{\tPP}{\tilde{\PP}_n}

\newcommand{\Ccal}{\mathcal{C}}
\newcommand{\Xcal}{\mathcal{X}}

\newcommand{\Fcal}{\mathcal{F}}
\newcommand{\Hcal}{\mathcal{H}}
\newcommand{\Rcal}{\mathcal{R}}
\newcommand{\Ncal}{\mathcal{N}}

\newcommand{\vol}{\mathrm{vol}}

\newcommand{\K}{\mathcal{K}^{(1)}_d}
\newcommand{\R}{\mathbb{R}}

\renewcommand{\eqref}[1]{Eq.~(\ref{#1})}

\newcommand{\TV}{\mathrm{TV}}

\startlocaldefs
\endlocaldefs

\begin{document}

\begin{frontmatter}

\title{Optimality of Maximum Likelihood for Log-Concave Density Estimation and Bounded Convex Regression}
\runtitle{The Optimality of MLE }

\author{\fnms{Gil} \snm{Kur}\ead[label=e1]{gilkur@mit.edu}}
\and
\author{\fnms{Yuval} \snm{Dagan}\ead[label=e2]{dagan@mit.edu}}
\and
\author{\fnms{
Alexander} \snm{Rakhlin} \ead[label=e3]{rakhlin@mit.edu}}
\affiliation{Massachusetts Institute of Technology}

\runauthor{Kur et al.}
\begin{abstract}
	In this paper, we study two problems: (1) estimation of a $d$-dimensional log-concave distribution and (2) bounded multivariate convex regression with random design with an underlying log-concave density or a compactly supported distribution with a continuous density. 
	First, we show that for all $d \ge 4$ the maximum likelihood estimators of both problems achieve an optimal risk of $\Theta_d(n^{-2/(d+1)})$\footnote{In the  regression setting, this bound is tight for certain measures, e.g. when the underlying distribution is uniform on a ball. However, for some log-concave measures, the minimax is of order $\Theta_d(n^{-\frac{4}{d+4}})$.} (up to a logarithmic factor) in terms of squared Hellinger distance and $L_2$ squared distance,  respectively. Previously, the optimality of both these estimators was known only for $d\le 3$. We also prove that the $\epsilon$-entropy numbers of the two aforementioned families are equal up to logarithmic factors. We complement these results by proving a sharp bound $\Theta_d(n^{-2/(d+4)})$ on the minimax rate (up to logarithmic factors) with respect to the total variation distance.  
	Finally,  we prove that estimating a log-concave density---even a uniform distribution on a convex set---up to a fixed accuracy requires the number of samples \emph{at least} exponential in the dimension. We do that by improving the dimensional constant in the best known lower bound for the minimax rate from $2^{-d}\cdot n^{-2/(d+1)}$ to $c\cdot n^{-2/(d+1)}$ (when $d\geq 2$).
\end{abstract}
\end{frontmatter}

\section{Introduction} 
\label{sec:1}
Consider the problem of estimating the regression function $f^*$ over the domain $\Xcal\subseteq\reals^d$, based on a set of real-valued observations $Y_i=f^*(X_i)+\xi_i$, $i=1,\ldots,n$, with the knowledge that $f^*\in\Fcal$. Here $X_1,\ldots,X_n\sim \PP$ are independent, and $\xi_1,\ldots,\xi_n$ are i.i.d. and zero mean. The Least Squares (LS) estimator 
$$\widehat{f} \in \argmin{f\in\Fcal}~ \sum_{i=1}^n (Y_i-f(X_i))^2$$
is perhaps the most basic procedure, and its optimality properties are of central interest. 

Similarly, the problem of density estimation corresponds to recovering the unknown density $f^*\in\Fcal$ based on $n$ independent observations from a distribution with this density. The most basic procedure, Maximum Likelihood Estimation (MLE), is defined as
$$\widehat{f} \in \argmax{f\in\Fcal}~ \sum_{i=1}^n \log f(X_i).$$

Analysis of LS and MLE for large non-parametric classes $\Fcal$ rests on uniform laws of large numbers, as quantified by the theory of empirical processes \citep{van2000empirical}. Classical results due to Le Cam and Birg\'e establish minimax rates for regression and density estimation based on notions of ``richness'' of $\Fcal$. In particular, the solution $\epsilon_{*}^2$ to 
\begin{align}
\label{minimax_fixed_pt_lecam}
    \frac{\log \Ncal_2(\Fcal, \epsilon_{*}, \PP)}{n} ~\asymp~ \epsilon_{*}^2
\end{align}
provides, under appropriate conditions, the global minimax rates for estimation with respect to squared $L_2(\PP)$ (for regression) and squared Hellinger (for density estimation) measures of closeness \citep{yang1999information}. Here $\Ncal_2(\Fcal, \epsilon, \PP)$ is a covering number of $\Fcal$ with respect to $L_2(\PP)$ at scale $\epsilon$, defined as the smallest number of functions $f_1,\ldots,f_N\in\Fcal$ such that $\forall f\in\Fcal, \exists j$ s.t. $\|f-f_j\|_{L_2(\PP)}\leq \epsilon$ . 

On the other hand, the work of \citep{birge1993rates} shows that upper bounds on the rate of MLE/LS can be derived by solving
\begin{align}
\label{eq:fixed_pt_dudley_old}
    \frac{1}{\sqrt{n}}\int_{0}^\epsilon \sqrt{\log \Ncal_2(\Fcal,\delta, \PP_n)} d\delta ~\asymp~ \epsilon^2,
\end{align}
where $\PP_n$ denotes the empirical measure with respect to i.i.d. samples $X_1,\ldots,X_n$ drawn from $\PP$. This technique is used to establish minimax optimality of MLE and LS, however it only works in the regime when the Dudley entropy integral in \eqref{eq:fixed_pt_dudley_old} is finite, meaning that the entropy  $\log \Ncal_2(\Fcal, \delta, \PP_n) = o(\delta^{-2})$. The complexity regime with converging integral is termed the ``Donsker regime,'' and the analysis of optimality of LS and MLE in the ``non-Donsker regime'' appears to be on a case-by-case basis. Existence of a general set of geometric conditions sufficient (and perhaps necessary) for establishing optimality of LS/MLE in the non-Donsker regime is still an open area of research in nonparametric statistics. 

There are a few approaches that can be attempted in the non-Donsker regime. First, the standard symmetrization-based analysis yields, under appropriate boundedness assumptions, a (potentially loose) upper bound on the estimation error for LS in terms of the \textit{global} Rademacher averages of $\Fcal$, defined below. One can estimate the global Rademacher averages via entropy numbers with respect to the empirical measure $\PP_n$, but estimates on theses numbers are not always available for the classes of interest. 
The following fixed point has also been used in the literature as an upper bound on the rate of estimation in LS/MLE \citep{kim2016global,han2016multivariate,gardner2006convergence}:
\begin{align}
\label{eq:fixed_pt_dudley_old_non}
    \frac{1}{\sqrt{n}}\int_{\epsilon}^1 \sqrt{\log \Ncal_{[],2} (\Fcal,\delta, \PP)} d\delta ~\asymp~ \epsilon.
\end{align}
Here $\log \Ncal_{[], p}(\Fcal, \delta,\PP)$ denotes the bracketing entropy with respect to $L_p(\PP)$, defined as the smallest number of pairs $(\ell_j, u_j)\in\Fcal\times\Fcal$, $j=1,\ldots,N$, such that $\forall f\in\Fcal, \exists j$ s.t. $\ell_j\leq f\leq u_j$ and $\|u_j-\ell_j\|_{L_p(\PP)} \leq \epsilon$.
One can view the fixed point \eqref{eq:fixed_pt_dudley_old_non} as a (potentially loose) upper bound on the global Rademacher averages of $\Fcal$. Indeed, in the non-Donsker regime, one would replace the $L_2(\PP)$ bracketing numbers with the empirical $L_2(\PP_n)$ entropy for sharp estimates. The fixed point in \eqref{eq:fixed_pt_dudley_old_non} is larger than the fixed point in Eq. (\ref{minimax_fixed_pt_lecam}). Yet, as observed in the seminal work of \cite{birge1993rates}, \eqref{eq:fixed_pt_dudley_old_non} \emph{can} be tight for MLE/LS. 

Motivated by the work of \cite{schuhmacher2010consistency,carpenter2018near}, we present a rather simple recipe that can yield optimality of MLE and LS for certain families of functions or distributions in the non-Donsker regime\footnote{We acknowledge the work by \cite{han2019global} which appeared few months after our initial manuscript became available on arXiv. From a recent personal communication with the author, some of his results were achieved in his PhD thesis that was available online before our initial manuscript. The author used techniques that are very similar to our approach. }, and then showcase the technique by proving minimax optimality of \emph{bounded} convex regression (when the underlying distribution is log-concave or compactly supported with continuous density\footnote{See Remarks~\ref{Rem1},\ref{Rem2} for more details.}) and log-concave density estimation for $d \geq 4$. Both problems are known to fall in the non-Donsker regime as soon as $d \geq 4$, and optimality of MLE/LS has remained an open problem.
More specifically, we prove that MLE for log-concave density estimation and LS for bounded convex regression  achieve the optimal rates of $O_d(n^{-\frac{2}{d+1}})$, up to a logarithmic factor, in squared Hellinger and squared $L_2$, respectively. 

Our technique is to reduce the MLE/LS problem to a question of uniform convergence over indicators of sets $\Ccal$, the level sets of functions in the class. We establish an upper bound on the expected supremum of the empirical process 
\begin{align}
    \label{eq:sup_emp_proc_C}
    \E\sup_{C\in\Ccal} \left| \PP_n(C)-\PP(C) \right|,
\end{align}
for classes $\Ccal$ beyond the Donsker regime. In contrast to the classical fixed point in \eqref{eq:fixed_pt_dudley_old}, the rate is given by $\epsilon^2$ that solves
\begin{align}
\label{eq:new_fixed_pt}
    \frac{1}{\sqrt{n}}\int_{\epsilon^2}^1 \sqrt{\frac{\log \Ncal_{[],1}(\Ccal,\delta,\PP)}{\delta}} d\delta ~\asymp~ \epsilon^2.
\end{align}

To illustrate the calculation, the bracketing entropy of a collection of convex subsets of a bounded set at scale $\delta$ can be upper bounded by $\Theta_d\left(\delta^{-\frac{d-1}{2}}\right)$. The above fixed point then yields $\epsilon^2 \asymp n^{-\frac{2}{d+1}}$, the optimal rate for log-concave density estimation and bounded convex regression (when the underlying distribution has an arbitrary log-concave density; see Remark \ref{Rem2}). 

The form of the Dudley entropy integral in \eqref{eq:new_fixed_pt} is derived by a chaining technique with Bernstein tail bounds. The maximal inequality for a collection of Binomial random variables---that is, the empirical process indexed by sets---has a mixture of the sub-Gaussian and sub-Exponential tails, and we ensure that the first tail is dominating by requiring 
$$\frac{\log \Ncal_{[],1}(\Ccal, \epsilon^2)}{n} \lesssim \sqrt{\frac{\epsilon^2 \log \Ncal_{[],1} (\Ccal, \epsilon^2)}{n}}. $$
The fixed point \eqref{eq:new_fixed_pt} is consistent with this requirement when the entropy has a polynomial growth in the non-Donsker regime.

By comparing the minimax rate given by  \eqref{minimax_fixed_pt_lecam} and the fixed point of \eqref{eq:new_fixed_pt} for the polynomial growth of entropy, we observe that the two rates match whenever 
\begin{equation}\label{relation}
    \log \Ncal_{[],1}(\Ccal, \epsilon^2,\PP) \asymp \log \Ncal_{2} (\Fcal, \epsilon, \PP). 
\end{equation}
Interestingly, when the relation in \eqref{relation} holds, it implies that the bound \eqref{eq:fixed_pt_dudley_old_non} with $L_2(\PP)$ bracketing numbers is loose. 
Indeed, Lemma \ref{lem:compare_Rad_avg} below implies that the Rademacher complexity of $\Fcal$ is $O(\epsilon_*^2)$, which is significantly smaller than the fixed point given by \eqref{eq:fixed_pt_dudley_old_non}. Thus, we deduce that the empirical entropy $\log \Ncal_{2} (\Fcal, \epsilon, \PP_n)$ \emph{must differ} from the population entropy $\log \Ncal_{2} (\Fcal, \epsilon, \PP)$ at scales above the fixed point of \eqref{eq:fixed_pt_dudley_old_non}. In fact, from Sudakov's minoration and Lemma~\ref{lem:compare_Rad_avg} below, 
$$\sup_{\alpha>0} \alpha \sqrt{\frac{\log \Ncal_2(\Fcal, \alpha, \PP_n)}{n}} \lesssim \epsilon_{*}^2$$
 hence, the two entropies already differ at  scales above $\epsilon_{*}$. 
To complement this,  \cite[Lemma 9]{rakhlin2017empirical} implies that $\Ncal_2(\Fcal, \alpha, \PP)$ and $\Ncal_2(\Fcal, \alpha, \PP_n)$ behave similarly above the level $c\epsilon_{*}$, whenever the global Rademacher complexity is $O(\epsilon_*^2)$. Thus, we conclude that whenever \eqref{relation} holds, the empirical and the population entropy numbers behave similarly the until $\epsilon_{*}$ (for a detailed discussion, see Section \ref{Sec:Disc}).


\section{Main Results}

We start with a few definitions. For two probability measures $P$ and $Q$ with densities $p$ and $q$, total variation distance is defined as $d_{\TV}(P,Q) = \frac{1}{2}\int_{\R^d} |p(x)-q(x)|dx = \sup_{C \subset \R^d}|p(C)-q(C)|$. The squared Hellinger metric is defined as
$h^2(P,Q) = \frac{1}{2}\int_{\R^d}\left(\sqrt{p(x)}-\sqrt{q(x)} \right)^2 dx$. It holds that  $h^2 \leq d_{\TV} \leq \sqrt{2 h^2}$. For an estimator $\widehat{f}$ of a density $f^*\in\Fcal$, we define
$$\Rcal_{h^2}(\widehat{f}, \Fcal) = \sup_{f^*\in\Fcal} \E h^2(\widehat{f}, f^*)$$
and define $\Rcal_{\TV}(\widehat{f}, \Fcal)$ accordingly. Here the expectation is with respect to the data $X_1,\ldots,X_n$ from a distribution $\PP$ with density $f^*$. Similarly, for the problem of regression, we define
$$\Rcal_{L^2_2(\PP)}(\widehat{f}, \Fcal) = \sup_{f^*\in\Fcal} \E \left\| \widehat{f}- f^*\right\|_{L_2(\PP)}^2$$
where
$$\left\|f\right\|_{L_2(\PP)}^2 = \int f^2(x) \PP(dx),$$ and the expectation is with respect to $(X_1,Y_1),\ldots,(X_n,Y_n)$. Assumptions on the noise variable $Y$ will be specified in each theorem separately.

We denote by $\mathcal{F}_{d}$ the set of log-concave distributions in $\R^d$ and by $\Hcal_{d,\Gamma}$ the set of convex functions in $\R^d$ bounded by $\Gamma$. We denote a generic estimator by $\bar{f}$, namely a function from $n$-samples to a function. We use the notation $c,C, c',c_1$ for absolute \emph{positive} constants which do not depend on the problem parameters. $ c(d),C(d),c_d,\ldots $ are constants that depend only on the dimension of samples. Uppercase $C$ is used for constants greater than $1$ and lowercase $c$ for constants less than $1$. Throughout the paper these constant \emph{may} change from line to line.

The first set of results addresses log-concave density estimation in Hellinger loss.

\begin{restatable}{theorem}{thmone}
	\label{thm:hellinger}
	Assume that $d\geq 4$. The MLE $\widehat{f}$ achieves the risk
	\[
	\Rcal_{h^2}(\widehat{f}, \Fcal_d) \leq  C(d) \cdot n^{-\frac{2}{d+1}} \cdot \log(n).
	\]
	Furthermore, there exists a universal constant $c > 0$ such that for all $d \geq 2$,
	\begin{equation}\label{Eq:boundsfordim}
	\inf_{\bar{f}}\Rcal_{h^2}(\bar{f}, \Fcal_d)\geq \inf_{\bar{f}}\Rcal_{h^2}(\bar{f}, \Ccal_d) \geq c \cdot n^{-\frac{2}{d+1}},
	\end{equation}
	where $\mathcal{C}_d$ denotes the uniform distributions over the convex sets.
\end{restatable}
\noindent We believe that the logarithmic factor in the first part of the theorem is redundant. 

We now turn to the problem of bounded convex regression.
\begin{restatable}{theorem}{thm:regression}
	\label{thm:regression}
	Assume that $d\geq 4$, $\PP$ is log-concave and $\xi$ has $2+\epsilon$ moment bounded by $L$ (for some $\epsilon > 0$). Then the following holds for the LS estimator $\widehat{f}$:
    \begin{equation}\label{Eq:rEG}
        	\Rcal_{L^2_2(\PP)} (\widehat{f}, \Hcal_{d,\Gamma}) =  O_{d}(\Gamma \cdot \max\{C_{\xi},\Gamma\} \cdot n^{-\frac{2}{d+1}}).
    \end{equation}
	    where $C_{\xi}:= \int_{0}^{\infty}\sqrt{\Pr(|\xi|>t)}dt \leq L$.
\end{restatable}
The proof of this theorem appears in Section~\ref{sec:convex-reg}.
\begin{remark}\label{Rem1}
	We can remove the assumption of $\PP$ being log-concave on $\R^d$ at the expense of assuming that $\PP$ is supported on a bounded set and has a continuous density function. Theorem \ref{thm:regression} then holds with constant $C(\PP,d)$. Unfortunately, our proof does not extend to fixed design. 
		
		The recent result in \cite[Theorem 1]{han2019convergence}, implies that we may reduce the bounded $2+\epsilon$ moment of $\xi$ to only  $\frac{d+1}{d-1}+ \epsilon$ bounded moment, for some $\epsilon > 0$. 
\end{remark}
\begin{remark}\label{Rem2}
    Observe that the bound in \eqref{Eq:rEG} is tight, and it is achieved, for example, when the underlying distribution is uniform on the $d$-dimensional Euclidean ball.
    However, this bound is not optimal for every log-concave distribution. For example, when the distribution is uniform on the $d$-dimensional cube, the MLE achieves a risk of at most $O_d(n^{-2/d})$, as was observed in \citep{han2016multivariate}. For a further discussion, see Section~\ref{Sec:Prior}.
\end{remark}
\begin{remark}
    For the uniform bound on the ball, we can find the explicit dimensional constant in \eqref{Eq:rEG}. Specifically, if $\Gamma =1$ and $\xi \sim N(0,1)$
    \[
                	\Rcal_{L^2_2(\textrm{Unif}(B_d))} (\widehat{f}, \Hcal_{d,1}) \leq  Cd  \cdot n^{-\frac{2}{d+1}}.
    \]
\end{remark}
Finally, we give optimal (up to logarithmic factors) bounds for the $\epsilon$-entropy numbers of the almost isotropic log-concave distributions. Let $\Fcal_{d,\tilde{I}}$ denote the set of almost isotropic log-concave densities, namely,  the class of log-concave distributions with norm of the mean bounded by small absolute constant, and covariance satisfying $I/2 < \Sigma < 2I$. Let $\Ncal_{h}$ and $\Ncal_{\TV}$ denote the $\epsilon$-covering with respect to the Hellinger and total variation distance (equiv., $L_1(\PP)$), respectively.
\begin{restatable}{theorem}{Thm:Entropy}
	\label{Thm:Entropy}
	For $d\geq 4$ and every $\epsilon > 0$, the following holds:
	\begin{enumerate}
	    \item $\Omega_d(\epsilon^{-(d-1)})\leq  \log \Ncal_{h}(\Fcal_{d,\tilde{I}}, \epsilon) \leq O_d(\epsilon^{-(d-1)}\log(\epsilon^{-1})^{(d+1)(d+2)/2}).$
	    \item $\Omega(\epsilon^{-d/2})\leq \log \Ncal_{\TV}(\Fcal_{d,\tilde{I}}, \epsilon) \leq O_d(\epsilon^{-d/2}\log(\epsilon^{-1})^{d(d+1)/2}).$
	\end{enumerate}
\end{restatable}
This theorem is proved in Section~\ref{sec:entropy_proof}.
\begin{remark}
    The first part of the Theorem is based on the following key results in \citep{gao2017entropy}:
    	\begin{enumerate}
	    \item $\log \Ncal_{[],2}(\Hcal_{d,1}, \epsilon,\mathrm{Unif}(B_d)) =\Theta_d(\epsilon^{-(d-1)})$
	    \item $\log \Ncal_{[],1}(\Hcal_{d,1}, \epsilon,\mathrm{Unif}(B_d)) =\Theta_d(\epsilon^{-d/2})$
	\end{enumerate}
	where $B_d$ denotes the $d$-dimensional Euclidean ball and $\mathrm{Unif}$ denotes the uniform distribution.
\end{remark}

We complement Theorem~\ref{thm:hellinger} with a corresponding result in the Total Variation distance. The proof (which  appears in Sections~\ref{sec:TV_lower} and  \ref{sec:TV_upper}), however, uses a technique different from that of  Theorems~\ref{thm:hellinger} and \ref{thm:regression}.
\begin{restatable}{theorem}{thmtwoupper}
	\label{thm:TV}
	Assume that $d \geq 2$, and $n$ is large enough. Then there exists an estimator $\bar{f}_{S}$, such that
	\begin{equation}\label{Eq:UppertvBound}
	\Rcal_{\TV}(\bar{f}_{S},\mathcal{F}_d) \leq  O_d(\log(n)^{\frac{d(d+1)}{d+4}}n^{-\frac{2}{d+4}}).
	\end{equation}
	Furthermore, there exists a universal constant $c > 0$, such that for all $d \geq 2$:
	\begin{equation}\label{Eq:lowerboundTV}
	\inf_{\bar{f}}\Rcal_{\TV}(\bar{f},\mathcal{F}_d) \geq c^{d} n^{-\frac{2}{d+4}}.
	\end{equation}
\end{restatable}
\noindent{This theorem shows that the minimax rate of learning a log-concave distribution with respect to total variation distance is $\Theta_d(n^{-2/(d+4)})$, up to a factor $\log(n)^{\frac{d(d+1)}{d+4}}$. Thus, it ``harder" to learn a log-concave distribution with respect to the total variation than the Hellinger squared distance.}

Theorems~\ref{thm:hellinger} and \ref{thm:regression} are based on the following result on empirical processes indexed by sets, proved in Section~\ref{sec:pr-final-chain}. 
\begin{lemma} \label{lem:final-chain}
	Let $\PP$ be a probability measure on $\mathbb{R}^d$. Let $\mathcal{A}$ be a family of measurable subsets of $\mathbb{R}^d$ with $\epsilon_0 =\PP(\bigcup \mathcal{A})$. Then for any  $\epsilon \ge 0$ such that $\log \mathcal{N}_{[],1}(\mathcal{A}, \epsilon,\PP) \le \epsilon n /3$,
\begin{equation} \label{eq:35}
\E\left[ \sup_{S \in \mathcal{A}}|\PP_n(S)-  \PP(S)| \right]
\le \epsilon +  \frac{C}{\sqrt{n}}\int_{\epsilon}^{\epsilon_0}\sqrt{\delta^{-1} \log \Ncal_{[],1}(\mathcal{A}, \delta,\PP)}d \delta.
\end{equation}
\end{lemma}

 The following corollary of Lemma~\ref{lem:final-chain} is the key ingredient in the proofs of upper bounds.
\begin{restatable}{corollary}{mainthm}
	\label{cor:conv-mcdiarmid}
	For any log-concave probability measure $\PP$, the following holds:\footnote{This result holds for more general distributions; see Lemma \ref{lem:conv-general}.} 
	\[
	\E\left[\sup_{C \in \mathcal{K}_d} \left| \PP_n(C) - \PP(C)\right| \right] \leq   O_d(n^{-2/(d+1)}),
	\] 
	where $\mathcal{K}_d$ denotes the collection of convex sets in $\R^d$.
\end{restatable}
	\begin{remark}\label{Rem:Lowerboundball}
	The  bound in Corollary \ref{cor:conv-mcdiarmid} is tight, for example, when $\PP= \mathrm{Unif}(B_d)$. Indeed, when $n$ is large enough, the convex hull of $n$ points $X_1\ldots, X_n$ chosen uniformly at random from $B_d$ satisfies the following \citep[Thm 1]{affentranger1991convex}:
	\[
	\frac{\E[\vol(B_d \setminus \mathrm{conv}\{X_1\ldots,X_n\})]}{\vol(B_d)} = \Theta(d\cdot n^{-\frac{2}{d+1}}).
	\]
	Hence,
	$
	    	\E\left[\sup_{C \in \mathcal{K}_d} |\PP_n(C) - \PP(C)|\right]
	\ge  c d \cdot n^{-\frac{2}{d+1}}.\footnote{The lower bound also follows from Sudakov's minoration inequality with a constant that is smaller than $d$, see \citep[Theorem 2.3]{han2019global}}
	$
	Remarkably, Lemma \ref{lem:conv-general} below shows that up to an \emph{absolute} constant this bound is tight.
	Based on this observation, and the classical results that show that the Euclidean ball  ``is the hardest to approximate" \citep{macbeath1951extremal,artstein2015asymptotic} (in various settings), we believe that the bound in Corollary~\ref{cor:conv-mcdiarmid} is $\Theta(d \cdot n^{-\frac{2}{d+1}})$.
	\end{remark}
\section{Prior Work}\label{Sec:Prior}

Log-concave distributions constitute a natural non-parametric family that includes the Gaussian, exponential, uniform over convex bodies, logistic, Gamma, Laplace, Weibull, Chi and Chi-Squared, Beta distributions and more. This rich family has a key role in statistics \citep{bagnoli2005log}, pure mathematics \citep{brazitikos2014geometry,stanley1989log}, computer science \citep{balcan2013active,lovasz2007geometry,axelrod2019polynomial} and economics \citep{an1997log}. The task of estimating a log-concave density has a long history (see, for example, the recent survey of \citet{samworth2018recent}), and the rates of convergence have been a focus of the literature in the last decade, with an incomplete list being \citep{diakonikolas2016learning,dumbgen2011approximation,dumbgen2009maximum,kim2016global,cule2010theoretical,carpenter2018near,schuhmacher2010consistency,feng2018adaptation,kim2018adaptation,han2016approximation,doss2016global}.  Similarly, convexity is a natural assumption on the shape of a regression function. The rates of convex regression in random and fixed design have been investigated in the literature, including the works of \citep{gao2017entropy,han2016multivariate,guntuboyina2012optimal,guntuboyina2013covering,gardner2006convergence,bellec2018sharp,ghosal2017univariate}.


\paragraph{Density Estimation in Hellinger Distance} \citet{kim2016global} proved a lower bound on the minimax rate with respect to the squared Hellinger distance
\begin{equation}\label{Eq:KimLowerBound}
\inf_{\bar{f}} \mathcal{R}_{h^2}(\bar{f}, \Fcal_d) \geq
 \begin{cases}
cn^{-\frac{4}{5}} & \text{if} \ d = 1  \\ e^{-cd}n^{-\frac{2}{d+1}} & \text{if} \ d\geq 2.
\end{cases} 	
\end{equation}

The authors proved that the log-concave MLE  $\widehat{f}$ achieves this lower bound when $d \leq 3$, up to a logarithmic factor:
\begin{equation}\label{Eq:KimMLEUpperBound}
\mathcal{R}_{h^2}(\widehat{f}, \Fcal_d) \le C \cdot 
\begin{cases}
n^{-\frac{4}{5}} & d = 1 \\ n^{-\frac{2}{d+1}}\log(n) & d = 2, 3.
\end{cases}
\end{equation}
Their proof, based on $\epsilon$-entropy numbers of  log-concave distributions, does not extend for $d \geq 4$, due to the aforementioned obstacle of the non-Donsker regime. 

In high dimensions ($d \geq 4$), \cite{schuhmacher2010consistency} showed that the MLE is  consistent, and recently \citet{carpenter2018near} were the first to achieve an upper bound of
\begin{equation}\label{Eq:CarpMLEUpperBound}
\mathcal{R}_{h^2}(\hat{f}) \le O_d\left(n^{-\frac{2}{d+3}}\log^3(n)\right) \quad \text{for} \  d \geq 4
\end{equation}
for its risk. Yet, this estimate is still sub-optimal, and it strengthened the conjecture of \citet{samworth2018recent} that the MLE does not achieve the minimax rate when $d \geq 4$. The first part of Theorem~\ref{thm:hellinger} in the present paper disproves this conjecture (up to a logarithmic factor). Our proof is based on an approach first introduced in a paper of \citet{schuhmacher2010consistency} in order to show a uniform consistency of the MLE, and later  extended by \citet{carpenter2018near} to show the non-asymptotic bound in \eqref{Eq:CarpMLEUpperBound}. 

The second part of Theorem \ref{thm:hellinger} improves the constant in \eqref{Eq:KimLowerBound} from $e^{-cd}$ to an absolute constant. Moreover, it also gives an improvement to the main result of \cite{RademacherG09} that studied the sample complexity of estimating a convex body.\footnote{This work is in a slightly different setting, known as the random oracle design. The authors showed that at least $2^{\Omega(\sqrt{d})}$ samples are required to estimate a convex set up to fixed accuracy. This work improves the bound to $2^{\Omega(d)}.$}

\paragraph{Density Estimation in Total Variation} In terms of  total variation, the minimax risk is only known for univariate log concave distributions \citep{chan2013learning,devroye2012combinatorial,diakonikolas2016efficient} and equals  $\Theta(n^{-2/5})$. However, in Hellinger squared distance the risk is different and equals $\Theta(n^{-4/5})$. In high dimensions ($ d \geq 2$), a difference between the minimax risks was not known, and the best lower bound in both metrics was $\Omega_d(n^{-2/(d+1)})$, as already presented in \eqref{Eq:KimLowerBound}.  \citet{diakonikolas2016learning} were the first to achieve an upper bound for $d \geq 4$, showing that there is an estimator that achieves risk upper bounded by  $O_d(n^{-2/(d+5)} \log(n)^{4/(d+5)} )$. In the aforementioned paper, the authors conjecture that minimax rate of this problem is $\Theta_d(n^{-2/(d+4)})$. Indeed,  Theorem~\ref{thm:TV} in the present paper we prove that the minimax rate of learning a log-concave density is $\Theta_d(n^{-\frac{2}{d+4}})$, up to a factor of $\log(n)^{\frac{d(d+1)}{d+4}}$.

\paragraph{Bounded Convex Regression}
\cite{han2016multivariate} analyzed the problem of bounded convex regression, and their results are based on the tight bounds for the $\epsilon$-entropy numbers of the bounded convex functions in \citep{gao2017entropy}. One of their results is a tight estimation of the minimax rate, in the cases that the underlying distribution is uniform measure on the Euclidean ball or on the cube (more generally, a polytope with a restricted number of facets). They proved that
\begin{equation}{\label{Eq:111}}
    \inf_{\bar{f}}\Rcal_{L_2^2(\mathrm{Unif}(B_d))}(\bar{f}, \Hcal_{d,\Gamma}) =  \Theta_{d}(C_{\Gamma,\xi} \cdot n^{-\frac{2}{d+1}}) \ , \ \inf_{\bar{f}}\Rcal_{L_2^2(\mathrm{Unif}([0,1]^d))}(\bar{f}, \Hcal_{d,\Gamma}) =  \Theta_{d}(C_{\Gamma,\xi} \cdot n^{-\frac{4}{d+4}}).
\end{equation}

However, for LS, \citep{han2016multivariate} only showed the following upper bounds
\[
	\Rcal_{L_2^2(\mathrm{Unif}(B_d))}(\hat{f}, \Hcal_{d,\Gamma}) \le C(d,L) \cdot \Gamma \cdot 
	\begin{cases}
	n^{-\frac{4}{5}} & d = 1 \\ n^{-\frac{2}{d+1}}\log(n) & d = 2, 3 \\ n^{-\frac{1}{d-1}} & d \geq 4
	\end{cases}
\]
and
\begin{equation}\label{Eq:2222}
	\Rcal_{L_2^2(\mathrm{Unif}([0,1]^d))}(\hat{f}, \Hcal_{d,\Gamma})\le C(d,L) \cdot \Gamma \cdot
	\begin{cases}
	n^{-\frac{4}{d+4}} & d \leq 4  \\ n^{-\frac{2}{d}} & d \geq 5
	\end{cases}
\end{equation}
under the assumption that $\xi$ is $L$ sub-exponential. Observe that when $d\geq 4$ the first bound differs from the minimax rate, and similarly the second bound for $d\geq 5$.
These results suggest that in the setting of bounded convex regression, \emph{the domain plays a key role}.
 Theorem~\ref{thm:regression} and \eqref{Eq:111} show that the LS is optimal when  $\PP = \mathrm{Unif}(B_d)$ and $d \geq 4$.
Moreover, the upper bound in Theorem~\ref{thm:regression} holds for any log-concave probability measure and $\xi$ that has $\frac{d+1}{d-1}+\epsilon$ moment bounded by $L$ for some $\epsilon >0 $. However, by \eqref{Eq:2222} this bound is not optimal for some log-concave densities.

\paragraph{Learning Convex Bodies}
The problem of estimating a convex body from the convex hull of random samples is extensively studied in the high dimensional geometry literature (see, for example, the books by \citet{schneider2008stochastic} and \citet{chiu2013stochastic}) and statistics \citep{brunel2013adaptive,brunel2016adaptive}. It is well-known that if the points are drawn uniformly \citep{barany1992random}, then with high probability the convex hull is $ O_d(n^{-2/(d+1)})$-close to the original set in symmetric volume difference.
This implies that the  risk of learning uniform distributions over convex bodies is $O_d(n^{-2/(d+1)})$ with respect to both the Hellinger squared and the total variation. Observe that the aforementioned family is a small subset of the log-concave distributions. Remarkably, our first result implies that the log-concave MLE achieves the same risk with respect to the \emph{squared Hellinger} distance (up to a logarithmic factor). In contrast, our second result shows that with respect to the total variation distance, learning a log-concave density is \emph{harder} than the uniform distributions of convex sets.

\paragraph{The $\epsilon$-entropy for the two problems} In this paper, we show that entropy numbers of almost isotropic log-concave distributions and bounded convex functions are equivalent up to logarithmic factors.
Theorem~\ref{Thm:Entropy} closes the gap of the Lemma in \citep{kim2016global}, they proved the following lower bound for the $\epsilon$-entropy numbers of almost isotropic log-concave distributions when $d \geq 4$  with respect to the Hellinger metric:
\[
	\log \mathcal{N}_{h}(\Fcal_{d,\tilde{I}},\epsilon) \geq \Omega_d(\epsilon^{-(d-1)})
\]

We show that, up to logarithmic factors, the same upper bound holds. Our proof is based on main result in \cite{gao2017entropy} who showed that  $\epsilon$-entropy numbers of the bounded convex functions supported on the Euclidean ball are equal to $\Theta_d(\epsilon^{-(d-1)})$.  We also give a sharp bound (up to logarithmic factors) for the $\epsilon$-entropy numbers with respect to the total variation distance.

\section*{Acknowledgements}
The authors want to thank Adityanand Guntuboyina, Ilias Diakonikolas, and Richard Samworth for their helpful comments. This material is based upon work supported by the Center for Brains, Minds and Machines (CBMM), funded by NSF STC award CCF-1231216.
\section{Discussion: employing the bracketing numbers in high dimensions}\label{Sec:Disc}

For both problems studied in this paper, directly using the $\epsilon$-bracketing numbers of these families together with standard chaining techniques leads to sub-optimal rates. In this section, we elaborate on the discussion in the Introduction to highlight these difficulties. 

To start the discussion, consider the result of \citet{gao2017entropy} for bracketing entropy of bounded convex functions
\begin{equation} \label{eq:16}
\log \mathcal{N}_{[],2}(\Hcal_{d,1},\epsilon,\textrm{Unif}(B_d)) = \Theta_d(\epsilon^{-(d-1)}).
\end{equation}
These covering numbers are on distances with respect to the uniform measure over $B_d$. Given $n$ samples, we have access to a projection $\Hcal_{d,1,n} := \{(h(X_1),\ldots,h(X_n)) : h \in \Hcal_{d,1}\}$ of the class onto the data.
The symmetrization-based arguments lead us to estimate the global/local Rademacher complexity of the set $\Hcal_{d,1,n}$, therefore we consider the empirical entropy numbers $\log \mathcal{N}_{2}(\Hcal_{d,1},\epsilon,\PP_n)$. 
Estimating the empirical $\epsilon$-entropy may be a hard task, therefore we naturally try to use the population bracketing entropy numbers (in this case, uniform), and this step is justified (in order to find a tight bound) if they behave similarly to the empirical entropy numbers until the fixed point in \eqref{eq:fixed_pt_dudley_old_non} (in which case, this fixed point is the global Rademacher complexity, up to a constant). As we explain below, in this case, the equivalence only holds for $\epsilon \geq C(d)n^{-\frac{1}{d+1}}$, while the fixed point of \eqref{eq:fixed_pt_dudley_old_non} is $C_1(d)n^{-\frac{1}{d-1}}$. 

First, the lower bound on the $\epsilon$-covering numbers with respect to the uniform measure \citep[Theorem~4]{gao2017entropy} does not hold when one considers the empirical measure instead. 
The proof is based on a construction of an $\epsilon$-separated set of functions that differ from each other by a set of $c(d)\epsilon^{-(d-1)}$ disjoint caps (see Definition \ref{def:cap}) with height of $\epsilon^2$. Clearly any two functions differ by a set of volume  $c(d)\epsilon^2$. Moreover, observe that all these caps lie in the shell
\[
	\{ x \in \R^d: 1-c(d)\epsilon^2 \leq \|x\|_2 \leq 1\}
\]  
that has a volume of $\Theta_{d}(\epsilon^2)$. In order to witness the difference between the functions with respect to the empirical measure, we must have at least one point in each cap for at least half of the caps. Hence, we need that $c(d)\epsilon^{-(d-1)}$ points fall in a set of measure $C(d)\epsilon^2$, which is $\Theta_{d}(\epsilon^2 n)$ under the uniform measure. Therefore, the following must hold:
\[
	c(d)\epsilon^{-(d-1)} \leq \epsilon^2n \iff \epsilon \geq c(d) n^{-\frac{1}{d+1}}
\]
To conclude, if $\epsilon < c(d)n^{-\frac{1}{d+1}}$, the empirical distances provide a ``distorted" picture as compared to the population distances. As discussed in the Introduction, since the level sets technique succeeds in this case, for any such $\epsilon \leq c(d)n^{-\frac{1}{d+1}}$ one cannot find any $\epsilon$-separated set (w.r.t. to the $L_2(\PP_n)$) of the aforementioned cardinality, that is $\Omega_d(\epsilon^{-(d-1)})$. However, for $\epsilon > C(d)n^{-\frac{1}{d+1}}$, the population and empirical  entropies  are equal up to some absolute constant. It follows from a classical result in empirical process theory (see \cite[Lemma 5.11]{van2000empirical}) (see also \cite{rakhlin2017empirical,bousquet2002concentration}), that connects the empirical and population distances of any family of functions. More precisely, up a multiplicative absolute constant, for any $f,g \in \Hcal_{d,1}$,
\begin{equation}\label{Eq:bounddistance}
\|f-g\|_{L_2(\PP_n)} \asymp \|f-g\|_{L_2(\PP)} + d^{*},
\end{equation}
where $d^{*}$ is the solution to $\log \mathcal{N}_{[],2}(\Hcal_{d,1}-\Hcal_{d,1},\epsilon,\textrm{Unif}(B_d)) = Cn\epsilon^{2}$. Using the fact that $\Hcal_{d,1}$ is convex and it contains the zero function, \eqref{eq:16} implies that $\log \mathcal{N}_{[],2}(\Hcal_{d,1}-\Hcal_{d,1},\epsilon,\textrm{Unif}(B_d)) \lesssim C(d)\epsilon^{-(d-1)}$, and therefore $d^{*} \asymp n^{-1/(d+1)}$.


This phenomenon of differing empirical and population entropy numbers above the fixed point in \eqref{eq:fixed_pt_dudley_old_non} \emph{does not always occur} in the non-Donsker regime. As an example, consider the family of distributions/functions constructed by \citet{birge1993rates} to show suboptimality of MLE for certain non-Donsker classes. For each $\alpha \in (0,1/2)$, the authors constructed a subset of functions that are $\alpha$-Holder continuous. Following their construction, the population $\epsilon$-entropy numbers are of the order $O(\epsilon^{-\frac{1}{\alpha}})$, and the empirical entropy numbers only differ (up to a constant that depends on $\alpha$) when $\epsilon \lesssim n^{-\alpha}$, that is the fixed point  \eqref{eq:fixed_pt_dudley_old_non}.

As mentioned before, our approach is to reduce the estimation problems to a question of uniform convergence over level sets. For bounded convex regression this reduction is established in Section~\ref{sec:convex-reg}, while the reduction for log-concave estimation was shown by \citet{carpenter2018near} (and earlier, an asymptotic  version in \cite{schuhmacher2010consistency}).

\section{Preliminaries and notations}

The volume (or a surface area) of a set is denoted by $\mathrm{vol}(\cdot)$. The centered  $d$-dimensional Euclidean ball with radius $1$ is denoted by $B_d$ and $\partial B_d$ denotes its surface. The Euclidean ball with center $x$ and radius $r$ is denoted by $B_d(x,r)$, and $ B_d(r) $ denotes a centered ball with radius $ r $.
The collection of all convex bodies contained in $\mathbb{R}^d$ is denoted by $\mathcal{K}_d$, and $\mathcal{K}^{(R)}_d$ denotes the collection of bodies contained in $ R \cdot B_d$.

Let $(\mathcal{M},d)$ be a metric space, fix $\epsilon>0$ and let $\mathcal{S} \subseteq \mathcal{M}$ be a finite subset. We say that $\mathcal{S}$ is an \emph{$\epsilon$-net ($\epsilon$-covering)} if for any $x \in \mathcal{M}$ there exists a $y \in \mathcal{S}$ such that $d(x,y) \le \epsilon$. The size of the smallest $\epsilon$-cover is called a covering number, and its logarithm is termed \textit{$\epsilon$-entropy}. 
	Given a partial ordering $\preceq$ over $\mathcal{M}$, we say that $\mathcal{S}$ is a \emph{$\epsilon$-net with bracketing} if for any $x \in \mathcal{M}$ there exists a $\overline{y}, \underline{y} \in \mathcal{S}$ such that $d(\underline{y},\overline{y}) \le \epsilon$ and $\underline{y} \preceq x \preceq \overline{y}$. Define the \emph{$\epsilon$-bracketing numbers} as the smallest cardinality of an $\epsilon$-net with bracketing (will be denoted by adding $[]$ as a subscript in  $\mathcal{N}$).
\subsection{Geometry}

For a point $p \in \mathbb{R}^d$ and a measurable set $K \subseteq \mathbb{R}^d$, define $d(K,p) = \inf_{u \in K} \| p - u \|_2$. The Hausdorff distance between two convex bodies $K_1$ and $K_2$ in $\mathbb{R}^d$ is defined as 
	\[
	d_H(K_1,K_2) = \max\left( \max_{p_1 \in K_1} d(K_2, p_1),~ \max_{p_2 \in K_2} d(K_1, p_2) \right).
	\]

\begin{definition}{\label{def:cap}}
	A \emph{cap} of a $d$-dimensional ball, with height $h \in (0,1)$ and a center $x_0  \in \partial B_d $, is defined as  
	\[
	C(x_0,h):= \{x \in B_d: x_0^{\top}x \geq 1-h \}.	
	\]
\end{definition}  

\begin{lemma}\label{Lem:EuclideanBallVolume}
	The volume of $d$-dimensional ball is the following:
	\[
	\vol(B_d) = \frac{\pi^{\frac{d}{2}}}{\Gamma(\frac{d}{2}+1)}= (1+O(d^{-1}))(\pi d)^{-0.5}\left(\frac{2\pi e}{d}\right)^{\frac{d}{2}}.
	\]
	Therefore, the radius of a ball with volume one is $ (1+O(\frac{\ln(d)}{d}))\sqrt{\frac{d}{2\pi e}} $. Additionally, the following identities hold:
	\begin{enumerate}
		\item For all $ r \in (1,\infty) $, $ \vol(B_d(r\sqrt{d/2\pi e})) \leq r^d \vol(B_d) $.
		\item $ c\sqrt{d} \leq \frac{\vol(B_d)}{\vol(B_{d-1})} \leq C\sqrt{d}$.
		\item $ \vol(\partial B_d) = d\vol(B_d) $.
	\end{enumerate}
\end{lemma}

The following is a bound by \citet{bronshtein1976varepsilon,dudley1999uniform} on the covering numbers of $\mathcal{K}^{(R)}_d$ under the volume metric.

\begin{lemma} \label{lem:bro}
	For any $0 < \epsilon <1$, 
	\[
	\log \mathcal{N}_{[],H}\left(\mathcal{K}_d^{(R)},\epsilon\right)
	\leq Cd^{5/2} (\epsilon/R)^{-\frac{d-1}{2}},
	\]
	where $\mathcal{N}_{[],H}$ denotes the bracketing numbers with respect to the Hausdorff distance.	Also when $\epsilon$ is small enough, the following holds 
		\[
		\log \mathcal{N}_{[],1}\left(\epsilon,\mathcal{K}_d^{(R)},\vol\right)
		\leq Cd^{(d+4)/2}\vol(B_d)^{(d-1)/2}\left(\epsilon/R^d\right)^{-\frac{d-1}{2}}.
		\]
\end{lemma}

Let us formally define log-concave distributions. We refer the reader to \citet{brazitikos2014geometry} for further details.
\begin{definition}
	A $d$-dimensional density function $f$ 
	is log-concave if the following holds for all $x , y \in \R^{d}$ and $\lambda \in (0,1)$:
	\[
	f(\lambda x + (1-\lambda)y) \geq f(x)^\lambda f(y)^{1-\lambda} \,.
	\]
	Equivalently, $f(x) = e^{-U(x)}$ where $U:\R^p \to \R\cup\{\infty\}$ is convex.
	We say that $f$ is an isotropic log-concave density if it has a zero mean and identity covariance.
\end{definition}

\begin{lemma}[\citep{kim2016global}] \label{lem:density-decay}
Let $f$ be an isotropic log-concave density on $\R^{d}.$ Then $f$ can be bounded by
\[
	f(x) \leq e^{-c_A(d)\|x\|_2 + C_B(d)},
\]
for positive dimensional constants $c_A(d)$ and $C_B(d)$.
\end{lemma}
The following is a variant of Assouad's  Cube Lemma  \citep{van2000asymptotic,tysbakovnon}
\begin{lemma}[Assouad's  lemma]{\label{Lem:Assoud}}
	Let $\mathcal{F}$ denote a family of functions. Fix $K \in \mathbb{N}$,
	and suppose that the family $\{f_{\alpha} \in \mathcal{F}: \alpha \in \{0,1\}^K \}$ has  the following two properties:
	\begin{enumerate}
		\item 	
		\[
		h^2(f_{\alpha},f_{\beta}) \geq \eta\|\alpha -\beta\|_0.
		\]
		for all $\alpha,\beta \in \{0,1\}^{K},$ where $\|\alpha-\beta\|_0$  denotes the Hamming distance between $\alpha$ and $\beta$
		\item There exists $ c \in (0,1)$ such that for every $\alpha,\beta \in \{0,1\}^{K}$ satisfying $ \|\alpha -\beta\|_0 = 1$, we have
		\[
		h^2(f_{\alpha},f_{\beta}) \leq \frac{c}{n}.
		\]
		Then
		\[
		\inf_{\bar{f}}\mathcal{R}_L(\bar{f}) \geq \frac{K}{8} \cdot (1-\sqrt{c})\cdot \eta.
		\]
	\end{enumerate}
\end{lemma}

\section{Proof of Theorem \ref{thm:regression}} \label{sec:convex-reg}
We use the notation $\PP_n(f) := \frac{1}{n}\sum_{i=1}^{n}f(X_i)$ and $\PP(f):= \E[f]$, where $f$ might be a level set of a convex function in the relevant context. The proof of the theorem rests on the following two results:
\begin{restatable}{lemma}{corr:noise}
	Assume that $d\geq 4$. Suppose $\xi$ has a $2+\epsilon $ moment bounded by $L$, and $\PP$ is a log-concave measure on $\R^d$. Then
	\label{corr:noise}	
	\[
	\E\left[\sup_{h \in \Hcal_{d,\Gamma}}\frac{1}{n}\sum_{i=1}^{n}h(X_i)\xi_i\right] \leq O_d(\Gamma \cdot C_{\xi} \cdot n^{-\frac{2}{d+1}}),
	\]
	where $C_{\xi}:= \int_{0}^{\infty}\sqrt{\Pr(|\xi|>t)}dt \leq L$.
\end{restatable}

\begin{restatable}{lemma}{corone}
	\label{cor:Gen}
	Assume that $d\geq 4$. For any distribution $\PP$ with a log-concave density,
	\[
	\E\left[\sup_{f,g \in \Hcal_{d,\Gamma}}\bigg| \PP_{n}(f-g)^2 -\PP(f-g)^2 \bigg|\right] \leq   O_d(\Gamma^2\cdot n^{-\frac{2}{d+1}}).
	\]
\end{restatable}

We can now prove the theorem. The following holds for the LS estimator:
\begin{align*}
&\mathrm{argmin}_{f \in \Hcal_{d,\Gamma}}\sum_{i=1}^{n}(f_0(X_i) + \xi_i - f(X_i))^2 = 
\\& \mathrm{argmin}_{f \in \Hcal_{d,\Gamma}} -2\sum_{i=1}^{n}(f_0(X_i) - f(X_i))\xi_i+\sum_{i=1}^{n}(f_0(X_i) - f(X_i))^2 
\end{align*}
Using the two Lemmas \ref{corr:noise} and \ref{cor:Gen}, we can bound the minimum in the last equation as follows: \footnote{The lemmas are proven in expectation. Using standard concentration inequalities they can be extended to high probability as well.}
\[
\argmin{f \in \Hcal_{d,\Gamma} } ~\PP(f_0-f)^2 + O_d(\Gamma \cdot \max\{C_{\xi},\Gamma\}  \cdot n^{-\frac{2}{d+1}}).
\]
It remains to prove Lemmas \ref{corr:noise} and  \ref{cor:Gen}.
\qed

\subsection{Proof of Lemma \ref{corr:noise}}
In order to prove the Lemma, we shall use the following Lemma (that is proven in the end of this-subsection):
\begin{lemma}
	\label{lem:compare_Rad_avg}
	Let $\Hcal = \{h:\R^d\to [0,\Gamma]\}$ be a class of non-negative bounded functions, and let $\Ccal$ be the corresponding collection of level sets. Assume $X_1,\ldots,X_n$ are i.i.d. from $\PP$. Then
	\begin{align}
	\E \sup_{h\in \Hcal} \frac{1}{n}\sum_{t=1}^n \epsilon_i h(X_i) \leq \Gamma\cdot \E \sup_{C\in\Ccal} |\PP_n(C) - \PP(C)| + \frac{C\cdot\Gamma}{\sqrt{n}}.
	\end{align}
\end{lemma}

It is enough to show that 
\[
	\E\left[\sup_{h \in \Hcal_{d,\Gamma}}\frac{1}{n}\sum_{i=1}^{n}\epsilon_i \cdot h(X_i)\right] \leq O_d(\Gamma \cdot n^{-\frac{2}{d+1}}),
\]
where $\epsilon_1,\ldots,\epsilon_n$ are i.i.d. independent Rademacher random variables. Indeed, by a standard result in empirical process (Lemma 2.9.1 in \citep{van1996weak}), it would follow that
\begin{equation}\label{Eq:radGau}
\E\left[\sup_{h \in \Hcal_{d,\Gamma}}\frac{1}{n}\sum_{i=1}^{n}\xi_i \cdot h(X_i)\right] \leq O_d(\Gamma \cdot C_{\xi} \cdot n^{-\frac{2}{d+1}}),
\end{equation}
where $C_{\xi}= \int_{0}^{\infty}\sqrt{\Pr(|\xi| > t)}dt$, (since the $2+\epsilon$ moment is bounded by $L$ we know that $C_\xi \leq L$).

The bound on Rademacher averages follows from Lemma~\ref{lem:compare_Rad_avg} and Corollary~\ref{cor:conv-mcdiarmid} since $h$ is convex on  $\R^{d}$ and $\PP$ is log-concave:
\begin{equation}\label{Eq:Tight}
    \E\left[\sup_{h \in \Hcal_{d,\Gamma}}\frac{1}{n}\sum_{i=1}^{n}\epsilon_i \cdot h(X_i)\right] \leq \int_{0}^{\Gamma}\E \sup_{h\in \Hcal_{d,\Gamma}}|\PP_n(h(x) \geq t) -\PP(h(x) \geq t)|dt 
\leq O_d(\Gamma \cdot n^{-\frac{2}{d+1}}). 
\end{equation}
All we need to to is to prove Lemma~\ref{lem:compare_Rad_avg}.
\subsection{Proof of Lemma~\ref{lem:compare_Rad_avg}}
	\begin{equation}\label{Eq:bla}
	\begin{aligned}
	\E\left[\sup_{h \in \Hcal}\frac{1}{n}\sum_{i=1}^{n}\epsilon_i h(X_i)\right]&=\E\left[\sup_{h \in \Hcal}\frac{1}{n}\sum_{i=1}^{n}\epsilon_i  (h(X_i)-\PP(h)+\PP(h))\right] 
	\\& \leq \E\left[|\sup_{h \in \Hcal}\frac{1}{n}\sum_{i=1}^{n}\epsilon_i  (h(X_i)-\PP(h))|\right] + \E \left[|\sup_{h \in \Hcal} \PP(h)\frac{1}{n}\sum_{i=1}^{n}\epsilon_i|\right] \\& \leq \E \left[\sup_{h \in \Hcal}|\mathbb{P}_n[h-\PP(h)]|\right] + C \cdot \Gamma \cdot n^{-0.5}, 
	\end{aligned}
	\end{equation}
	by the standard de-symmetrization argument (Lemma 2.3.6 in \citep{van1996weak}) and the fact that $h$ is nonnegative and bounded by $\Gamma$. Next, using the tail formula
	\begin{align*}
	\E \left[\sup_{h\in \Hcal}\big|\PP_n(h) -\PP(h)\big|\right] 
	&=\E\left[\sup_{h\in \Hcal}\big|\int_{0}^{\Gamma} \left(\PP_n(h(x) \geq t) -\PP(h(x) \geq t) \right)dt \big|\right]  \\
	&\leq  \int_{0}^{\Gamma} \E \sup_{h\in \Hcal} |\PP_n(h(x) \geq t) -\PP(h(x) \geq t)| dt
	\\& \leq \Gamma\cdot \E \sup_{C\in\Ccal} |\PP_n(C) - \PP(C)|,
	\end{align*}
	and the claim follows by the last two equations.
\subsection{Proof of Lemma \ref{cor:Gen}}
Since the functions in $\Hcal_{d,\Gamma}$ are uniformly bounded $\Gamma$, the contraction and  symmetrization Lemmas (see  \citep[Thms  2.1,2.3]{koltchinskii2011oracle})  imply that
\begin{align*}
\E\left[\sup_{f,g \in \Hcal_{d,\Gamma}}\bigg| \PP_{n}(f-g)^2 -\PP(f-g)^2 \bigg|\right] &\leq 8\Gamma\cdot \E\left[\sup_{f - g \in \Hcal_{d,\Gamma}-\Hcal_{d,\Gamma}}\frac{1}{n}\sum_{i=1}^{n}\epsilon_i \cdot (f-g)(X_i)\right]
\\& \leq 16 \Gamma \cdot \E\left[\sup_{h \in \Hcal_{d,\Gamma}}\frac{1}{n}\sum_{i=1}^{n}\epsilon_i h(X_i)\right].
\end{align*}
By Lemma~\ref{lem:compare_Rad_avg} and Corollary~\ref{cor:conv-mcdiarmid} , the last term is bounded by $O_d(\Gamma^2 \cdot n^{-\frac{2}{d+1}})$, and the lemma follows. We also provide an alternative proof to this Lemma that uses only Corollary~\ref{cor:conv-mcdiarmid} (see sub-section~\ref{S-Sup:cor:Gen} in the Supplementary).
\section{Proof of Theorem \ref{thm:hellinger}, Upper Bound} \label{sec:hellinger_upper}
\citet{carpenter2018near} used the properties of the log-concave distributions and squared Hellinger loss (Section~3 in their paper) and showed that the MLE risk is upper bounded by
\begin{equation}\label{Eq:KingDiako}
\Rcal_{h^2}(\hat{f}) \le C(d)\cdot \log(n) \cdot \E \left[\sup_{C \in \mathcal{K}_d} \left| \PP_n(C) - \PP(C)\right| \right],
\end{equation}
where the expectation is over the draw of $n$ i.i.d. samples from $\PP$. From Corollary \ref{cor:conv-mcdiarmid}, the expected supremum is $\Theta_d(n^{-2/(d+1)})$, and the rate in the theorem follows.

\citet{carpenter2018near} achieved a rate of $O_d(n^{-\frac{2}{d+3}}\log(n)^2)$ by utilizing the VC-dimension of polytopes with restricted number of facets. Our improvement in the bound follows from a different proof technique: We use the $\epsilon$-bracketing numbers of the convex sets under the volume metric, chaining methods, and Lemma~\ref{lem:density-decay} which certifies that isotropic log-concave have a ``fast" enough decay. For completeness, we give the sketch of the proof of \eqref{Eq:KingDiako} in Section~\ref{sec:log-concave-pr} in the Supplementary.

\section{Proof of Lemma~\ref{lem:final-chain}} \label{sec:pr-final-chain}

Given some measure $\PP$ over $\mathbb{R}^d$ and a family $\mathcal{A}$ of subsets of $\mathbb{R}^d$, we will bound 
\[
\E\sup_{S \in \mathcal{A}} |\PP_n(S) - \PP(S)|
\]
in terms of the bracketing numbers of $\mathcal{A}$ with respect to the natural metric induced by $\PP$ on sets, namely the measure of the symmetric difference $\PP(S \triangle S')$ of two sets $S,S'\subset \R^d$  (equivalently, $L_1(\PP)$ on sets).  The derived bound is a variant of Dudley's integral, proved using the chaining method on ``sub-Exponential'' random variables. Let $r \ge 0$ be an integer and assume that $\log \mathcal{N}_{[],1}(\mathcal{A},2^{-r},\PP) \le 2^{-r} n /3$. We will now prove that 
\begin{equation} \label{eq:36}
\E\left[ \sup_{S \in \mathcal{A}}|\PP_n(S)-  \PP(S)| \right]
\le 2^{-r} + C \sum_{i=r_0}^{r}
n^{-0.5}\sqrt{2^{-i} \log \mathcal{N}_{[],1}(\mathcal{A},2^{-r},\PP)},
\end{equation}
where $r_0$ be the smallest integer $i$ such that $2^{-i} < \PP(\bigcup \mathcal{A})$.

For convenience, denote $\tPP := \PP_n - \PP$. Let $\mathcal{N}_i$ be the minimal net $2^{-i}$-net with bracketing for $\mathcal{A}$ with respect to $L_1(\PP)$. Fix some $S \in \mathcal{A}$ and let $\overline{S},\underline{S} \in \mathcal{N}_r$ be elements such that $\underline{S} \subseteq S \subseteq \overline{S}$ and $\PP(\overline{S} \setminus \underline{S}) \le 2^{-r}$. Then
\[
\tPP(S) = \PP_n(S) - \PP(S) 
= (\PP_n(\overline{S}) - \PP(\overline{S})) + (\PP(\overline{S}) - \PP(S)) 
\le \max_{S \in \mathcal{N}_r} |\tPP(S)| + 2^{-r}.
\]
Similarly,
\[
- \tPP(S) = (\PP(S) - \PP_n(S)) \le (\PP(S) - \PP(\underline{S})) + (\PP(\underline{S}) - \PP_n(\underline{S})) \le  2^{-r} + \max_{S \in \mathcal{N}_r} |\tPP(S)|.
\]
We conclude that 
\begin{equation}\label{eq:78}
\E \sup_{S\in \mathcal{A}} |\tPP(S)|
\le \E \max_{S\in \mathcal{N}_r} |\tPP(S)| + 2^{-r}.
\end{equation}

We can assume that $\mathcal{N}_{r_0-1} = \{ \emptyset\}$. Indeed, from the conditions of this lemma, any $S \in \mathcal{A}$ satisfies $\PP(S \triangle \emptyset) = \PP(S) \le \PP(\bigcup \mathcal{A}) \le 2^{-r_0+1}$,  hence $\{\emptyset\}$ is a $2^{-r_0-1}$-net as required. As in a standard chaining argument, we write $\tPP(S)$ as a telescoping sum, with respect to elements in the different epsilon nets: let $\pi_i(S) = \arg\min_{S' \in \mathcal{N}_i} \PP(S \triangle S')$, and note that for any $S \in \mathcal{N}_r$,
	\[
	\tPP(S)
	= \tPP(\pi_r(S)) - \tPP(\pi_{r_0-1}(S))
	= \sum_{i = r_0}^r \left(\tPP(\pi_i(S)) - \tPP(\pi_{i-1}(S))\right),
	\]
where the first equality follows from the fact that $\pi_r(S) = S$ since $S \in \mathcal{N}_r$ and that $\pi_{r_0-1}(S) = \emptyset$.
This enables us to bound the maximum over $\mathcal{N}_r$ in terms of differences between elements of consecutive nets:

\begin{equation}
\begin{aligned} \label{eq:33}
\E\left[ \max_{S \in \mathcal{N}_r} |\tPP(S)| \right]
&= \E\left[ \max_{S \in \mathcal{N}_r}  \left| \sum_{i=r_0}^r \tPP(\pi_i(S)) - \tPP(\pi_{i-1}(S))\right|\right]
\\&\le \sum_{i=r_0}^r
\E\left[\max_{S \in \mathcal{N}_r}\left| \tPP(\pi_i(S)\triangle \pi_{i-1}(S)) \right|\right] 
\end{aligned}
\end{equation}
Note that the right hand side of \eqref{eq:33}, which is composed of a maximum of variables $\tPP(U)$ for sets $U \in \mathbb{R}^d$. 
Note that for each set $U$, $\PP_n(U)$ is a binomial random variable $\Bin(n, \PP(U))$. Depending on the size of the deviation relative to $n$ and $p$, a binomial variable can exhibit sub-Gaussian or sub-Exponential behavior. By balancing the two tails (see e.g. \citep[Corollary 2.6]{boucheron2013concentration}), we obtain the following lemma, under the condition that the number of variables is not too large:
\begin{lemma}\label{lem:exp-sup}
	Fix $0 < p < 1$ and  $n,k$ are positive integers. Assume that $Y_1, \dots, Y_k$ are random variables such that $Y_i \sim \Bin(n,p_i)$ for $p_i \le p$. Then, if $k \ge 2$ and $\log k \le n p/3$, then
	\begin{equation} \label{eq:2}
	\E\left[ \max_{1 \le i \le k} |Y_i/n - p_i| \right]
	\le C \sqrt{\frac{p \log k}{n}}.
	\end{equation}
\end{lemma}
Next, we bound each summand of the right hand side of \eqref{eq:33}. Fix $i = r_0, \dots, r$ and $S \in \mathcal{A}$, note that 
$n\PP_n(\pi_i(S)\triangle \pi_{i-1}(S)) \sim \Bin(n,p)$, with
\[
p := \PP(\pi_i(S) \triangle \pi_{i-1}(S)) 
\le \PP(\pi_i(S) \triangle S) + \PP(S  \triangle \pi_{i-1}(S))
\le 2^{-i} + 2^{-i+1}
= 3 \cdot 2^{-i}.
\]
Additionally, the maximum is over at most $\mathcal{N}_{[],1}(\mathcal{A},2^{-i},\PP) \cdot \mathcal{N}_{[],1}(\mathcal{A},2^{-i+1},\PP)$ elements, hence we can apply 
Lemma~\ref{lem:exp-sup} with $p = 3 \cdot 2^{-i}$ and $k \le \mathcal{N}_{[],1}(\mathcal{A},2^{-i},\PP) \mathcal{N}_{[],1}(\mathcal{A},2^{-i+1},\PP)$ to bound this maximum. Note that Lemma~\ref{lem:exp-sup} requires $\log k \le np/3$, and indeed using the assumption of this Lemma we know that
	$ \log \mathcal{N}_{[],1}(\mathcal{A},2^{-i},\PP) \le 2^{-i} n /3$.
	Hence,
	\[
	\log k
	\le \log \mathcal{N}_{[],1}(\mathcal{A},2^{-i},\PP) + \log \mathcal{N}_{[],1}(\mathcal{A},2^{-i+1},\PP)
	\le 2 \cdot 2^{-i}n /3
	\le np/3,
	\]
	substituting $p =3 \cdot 2^{-i}$.
	Applying Lemma~\ref{lem:exp-sup}, one obtains that for each $i = r_0,\dots,r$:
	\begin{align*}
	\E\left[ \max_{S \in \mathcal{N}_{r}} \left|\tPP(\pi_i(S)\triangle \pi_{i-1}(S))\right| \right] \le C \sqrt{2^{-i} \log \mathcal{N}_{[],1}(\mathcal{A},2^{-i},\PP)/n}.	    
	\end{align*}
	 We sum over $ i = r_0,\dots,r $ and conclude that  
	\[
	\E\left[ \max_{S \in \mathcal{N}_r} |\tPP(S)| \right]
	\le \sum_{i=r_0}^r 2C \sqrt{2^{-i} \log \mathcal{N}_{[],1}(\mathcal{A},2^{-i},\PP)/n}.
	\]
Combining with 	\eqref{eq:78}, then \eqref{eq:36} follows. 

\section{Proof of Corollary ~\ref{cor:conv-mcdiarmid}}
First, we state a lemma for arbitrary probability measures with a continuous density function.
\begin{lemma} \label{lem:conv-general}
	Let $\PP = f(x)dx$ be a probability measure with a continuous density function, and for any $r \ge 0$, let $M_r(f) := \sup_{x \colon \|x\| \ge r}  f(x) \cdot d \cdot \vol(B_d)$ denote a bound on the tail of $f$. Let $X_1, \dots, X_n \sim \PP$ be i.i.d. samples. Then,
	\begin{equation} \label{eq:21}
	\E\left[\sup_{K \in \mathcal{K}_d} |\PP_n(K) - \PP(K)|\right]
	\le  C \cdot \left(\sum_{i=0}^\infty 
	M_i(f)^{\frac{d-1}{d+1}} (i+1)^\frac{d(d-1)}{d+1}\right) \cdot n^{-\frac{2}{d+1}}.
	\end{equation}
	
    If we further assume that $\mathrm{supp}(\PP) \subset R \cdot B_d$, for some $R > 0$. Then, the following holds:
    	\begin{equation} \label{eq:211}
	\E\left[\sup_{K \in \mathcal{K}_d} |\PP_n(K) - \PP(K)|\right]
	\le  C(R^dM_0(f))^{\frac{d-1}{d+1}}n^{-\frac{2}{d+1}}.
	\end{equation}
\end{lemma}
First, we show that the Corollary \ref{cor:conv-mcdiarmid} follows from the aforementioned lemma. In the following sub-sections we will prove Lemma~\ref{lem:conv-general}. 

\begin{proof}[Proof of Corollary ~\ref{cor:conv-mcdiarmid}]
	First, if $\PP$ is isotropic log-concave, then, by Lemma~\ref{lem:density-decay}, 
	$f(x) \le e^{-c(d)\|x\|_2 + C(d)}$, and the claim follows by Lemma~\ref{lem:conv-general}. Now assume that $\PP$ is not isotropic, and that the covariance of $f$ is invertible (and a similar proof holds when it is singular).
	Then, there exists an invertible linear transformation $A$, such that $\PP^A$ is isotropic, where $\PP^A(S) := \PP(A^{-1}(S))$ for any measurable set $S \subset \R^d$. 
	Note that $\mathcal{K}_d = \{ T(K) \colon K \in \mathcal{K}_d\}$ for any invertible linear transformation $T$. Hence,
	\begin{align*}
	\E \sup_{K \in \mathcal{K}_d} | \PP_n(K) - \PP(K)|&
	= \E \sup_{K \in \mathcal{K}_d} | \PP^A_n(A^{-1}(K)) - \PP^A(A^{-1}(K))|
	= \E \sup_{K' \in \mathcal{K}_d} | \PP^A_n(K') - \PP^A(K')|
	\\&= O_d(n^{-2/(d+1)}),
	\end{align*}
	where the last step follows from the fact that $\PP^A$ is isotropic log-concave.  
\end{proof}
\begin{proof}[Proof of Lemma \ref{lem:conv-general}]\label{sec:conclude-main-lemma}
Note that we would like to bound the supremum of $\tPP(S)$ for any convex set $S$ which is not necessarily bounded. However, it would be convenient to consider bounded convex sets. In particular, we define the disjoint sets $A_0,A_1,A_2$ that satisfy $\bigcup_{i=0}^\infty A_i = \mathbb{R}^d$: $A_i := \{ x\in \mathbb{R}^d \colon i \le \|x\|_2 < i+1 \}$. We will write each convex set $i$ as a disjoint union of its intersections with the above sets $S = \bigcup_{i=0}^\infty S \cap A_i$. We will bound the supremum corresponding to each $A_i$ separately:	\begin{equation} \label{eq:3}
\E \sup_{K\in \mathcal{K}_d} |\tPP(K)|
\le \sum_{i=0}^\infty \E\sup_{K\in \mathcal{K}_d} |\tPP(K \cap A_i)|.
\end{equation}
We bound each summand, by using the following lemma that bounds $\E \sup_{S\in \mathcal{K}_d} |\tPP(S \cap A)|$ for a general bounded set $A$. Denote by $C_N(d)$ the dimensional constant of the bound on $\log \mathcal{N}_{[],1}(\K,\epsilon, \vol)$, namely, the minimal number such that $\log \mathcal{N}_{[],1}(\K,\epsilon,\vol) \le C_N(d) \epsilon^{-(d-1)/2}$ for all $0 < \epsilon \le 1$.

\begin{lemma} \label{lem:one-shell}
	Assume that $\PP = f(x)dx$ where $f$ is an arbitrary density function. Let $A$ be a bounded measurable set, and let $R = \sup_{x \in A} \|x\|_2$ and $M = \sup_{x \in A} f(x)$. 
	Then,
	\begin{equation}\label{eq:1}
	\E\left[\sup_{C \in \mathcal{K}_d} \left| \PP_{n}(C \cap A) - \PP(C \cap A) \right| \right]
	\le C \cdot \left( \frac{C_N(d)}{n}\right)^{\frac{2}{d+1}} (R^d M)^{\frac{d-1}{d+1}}.
	\end{equation}
\end{lemma}

Finally, we can prove Lemma~\ref{lem:conv-general}. 
	For any set $A_i,  i \ge 0$, we can apply Lemma~\ref{lem:one-shell} with $R = (i+1)$ and $M = M_i(f) = \sup_{x \colon \| x \| \ge i} f(x)$, and obtain that
	\[
	\E \sup_{K\in \mathcal{K}_d} |\tPP(K \cap A_i)|
	\le C \cdot \left( \frac{C_N(d)}{n}\right)^{\frac{2}{d+1}} (i+1)^{\frac{d(d-1)}{d+1}} M_i(f)^{\frac{d-1}{d+1}}.
	\]
	By \eqref{eq:3}, we derive that 
	\[
	\E \sup_{K\in \mathcal{K}_d} |\tPP(K)| \leq C \cdot \left( \frac{C_N(d)}{n}\right)^{\frac{2}{d+1}} 
	\sum_{i=0}^\infty 
	M_i(f)^{\frac{d-1}{d+1}} (i+1)^\frac{d(d-1)}{d+1}.
	\]
	Using the fact $C_N(d) \leq C_1 d^{(d+4)/2} \vol(B_d)^{(d-1)/2}$,  the lemma follows. Our last debt is  to prove  Lemma~\ref{lem:one-shell}.
\end{proof}
\begin{proof}[Proof of Lemma~\ref{lem:one-shell}]
	
	First, we will assume that $R = 1$, and then we will extend to arbitrary values of $R$. In particular, it suffices to consider only convex sets contained in $B_d$.
	Define $\mathcal{A}= \{ A \cap K \colon K \in \K \}$, and it holds that
	\begin{equation} \label{eq:331}
	\E \sup_{C \in \mathcal{K}_d} |\tPP(C \cap A)|
	= \E \sup_{S \in \mathcal{A}} |\tPP(S)|.
	\end{equation}
	We would like to apply Lemma~\ref{lem:final-chain}: For this purpose, we  use the following bound on the bracketing numbers of $\mathcal{A}$: for any $0 < \epsilon \le \PP(\cup\mathcal{A})$,
	\begin{equation} \label{eq:333}
	\log \mathcal{N}_{[],1}(\mathcal{A},\epsilon, \PP) 
	\le \log \mathcal{N}_{[], 1}(\K,\epsilon/M, \vol)
	\le C_N(d) \left( \frac{M}{\epsilon} \right)^{\frac{d-1}{2}},
	\end{equation}
	where the first part of the inequality follows from the fact that for any $S,S' \in \mathcal{A}$, $\PP(S \triangle S') \le \vol(S \triangle S') \sup_{x \in A} f(x) = M \vol(S \triangle S')$, and the second part follows Bronstein's lemma (see Lemma \ref{lem:bro}).
	We set $\epsilon_0 = \PP(\cup\mathcal{A})$ and $\epsilon$  that will be define later. Then, Lemma~\ref{lem:final-chain} implies
\begin{equation} \label{eq:111}
\begin{aligned}
\E \sup_{S\in \mathcal{A}} |\tPP(S)|
&\le \epsilon + Cn^{-0.5}\int_{\epsilon}^{\epsilon_0}\sqrt{t^{-1}\log \mathcal{N}_{[],1}(\mathcal{A},\epsilon, \PP)}dt \\
&\le \epsilon + C \cdot (n/C_N(d))^{-0.5}M^{\frac{d-1}{2}}\int_{\epsilon}^{\epsilon_0}t^{-(d+1)/4}dt\\
&\le \epsilon +  C \cdot (n/C_N(d))^{-0.5}M^{\frac{d-1}{2}}\epsilon^{-(d-3)/4},
\end{aligned}
\end{equation}	
	where the last inequality follows from the fact that $d \geq  4$. 
	We equalize the two terms of the right hand side of \eqref{eq:111}, and derive that $\epsilon \approx \left( \frac{ C_N(d)}{n}\right)^{\frac{2}{d+1}} M^{\frac{d-1}{d+1}}$.  
	Formally, $\epsilon$ is a number inside an interval that is uniquely defined by a unique integer, denoted by $r$, such that
	\[ 2^{-r-1} <  \left( \frac{C_N(d)}{n}\right)^{\frac{2}{d+1}} M^{\frac{d-1}{d+1}} \le 2^{-r}. \]
	Using this $\epsilon$, we obtain that the right hand side of \eqref{eq:111} is bounded by $C \left( \frac{C_N(d)}{n}\right)^{\frac{2}{d+1}} M^{\frac{d-1}{d+1}}$ as required. It remains to verify that we can apply Lemma~\ref{lem:final-chain} for this $\epsilon$, namely that $$\log \mathcal{N}_{[],1}(\mathcal{A},\epsilon, \PP) \leq C_N(d) \left( \frac{M}{\epsilon} \right)^{\frac{d-1}{2}}$$ Indeed, our $\epsilon$ is the smallest one (up to an absolute constant) that can be chosen that satisfies this inequality.
	
	Previously, we assumed that $R = \sup_{x \in A} \|x\| = 1$, and we will now remove this assumption. Assume that $R$ is arbitrary, Then rescale $x \mapsto x/R$.
	Formally, denote by the rescaled distribution $\PP'$ by $R$, namely that has the density $f'(y) = R^d f(Ry)$. Define the set $A' = A / R$, and note that
	\begin{equation} 
	\begin{aligned}
	\label{eq:22}
		\E \sup_{K \in \mathcal{K}_d} |\tPP(K \cap A)| 
		&= \E\sup_{K \in \mathcal{K}_d} |\tPP'(K/R \cap A/R)|
		= \E\sup_{K' \in \mathcal{K}_d} |\tPP'(K' \cap A')|\\
		&\le C \cdot \left( \frac{C_N(d)}{n}\right)^{\frac{2}{d+1}} (M')^{\frac{d-1}{d+1}},
	\end{aligned}
	\end{equation}
	where $M' = \sup_{x \in A'} f'(x)$ and the last inequality follows from applying this lemma on $\PP'$ and $A'$.
	Lastly, substituting $M' = R^d M$ in the right hand side of \eqref{eq:22}, the bound follows.
\end{proof}

\section{Proof of Theorem \ref{thm:hellinger}, Lower Bound}

Before we prove the theorem, we note that the proof builds upon the proof of Theorem 1 part (ii) in the paper by \cite{kim2016global}. Their proof relies on a packing argument of the unit sphere to generate a family of log-concave distributions. Any packing of the unit sphere can cover at most $e^{-cd}$ of its surface area, and this caused their proof to suffer from a similar factor in the risk.
Instead, we build an \emph{``approximate-packing''} which covers a constant fraction of the sphere. It allows us to create a large enough family of disjoint distributions to attain the desired bound. 

In order to simplify the proof, we first show how the following main lemma implies the theorem, and then we prove the lemma.
\begin{lemma}\label{Lem:Kur}
	Let $N \geq 10^{d}$ and $t_{d,N} \in (0.5,1)$, then there is a set of caps $C(x_1,t_{d,N}),\ldots, C(x_{cN},t_{d,N}) \subset B_d$ (see Definition \ref{def:cap}) that sasifies the following:
	\begin{enumerate}
	    \item $\vol(C(x_1,t_{d,N}))=\ldots = \vol(C(x_{cN},t_{d,N})) = c_{d,N}N^{-\left(1+\frac{2}{d-1}\right)}$.
	    \item 	For all $1 \leq i \leq cN$, the following holds:
	    \begin{equation}\label{Eq:Caps}
		  \vol\left( C(x_i,t_{d,N}) \setminus \bigcup_{j=1, i \ne j}^{cN}C(x_j,t_{d,N})\right) \geq c_3\vol(C(x_i,t_{d,N})) \geq c_4N^{-\left(1+\frac{2}{d-1}\right)}\vol(B_d).
\end{equation}
	\end{enumerate}
	where $c \in (0,1)$, $c_{d,N} \in (c_1,C_1)$.
\end{lemma}
Observe that the \eqref{Eq:Caps} implies that the overlap between the caps is not too large. Morever, these caps define a body 
	\[
	P_{d,N} :=  B_d \setminus \bigcup_{i=1}^{cN}C(x_i,t_{d,N})
	\]
	with volume
	\begin{equation}\label{Eq:volPd}
		\mathrm{vol}(P_{d,N}) = \left(1-C_{d,N}N^{-\frac{2}{d-1}}\right)\mathrm{vol}(B_d).
	\end{equation}
where $C_{d,N} \in (c_1,C_1)$.
The value of $N$ will be defined later, and we set $ K := cN$, where $c$ is the constant in Lemma~\ref{Lem:Kur}. Define $$\mathcal{F}_{\{0,1\}^{K}} :=  \left\{f_\alpha \in \mathcal{C}_d: \alpha \in \{0,1\}^{K} \right\},$$ where $f_{\alpha}$ is defined as 
\[
	f_{\alpha}(x) := \frac{\mathbbm{1}_{P_{d,N}}(x) + \mathbbm{1}_{\cup_{i=1}^{K}C(x_i,t_{d,N})\mathbbm{1}_{\alpha_i =1 }}(x)}{\vol(B_d)(1- C_{\alpha})},
\]
where $ \vol(B_d)(1- C_{\alpha}) = \vol(\bigcup_{i=1}^{K}C(x_i,t_{d,N})\mathbbm{1}_{\alpha_i =1 } \cup P_{d,N} )$ is the normalization factor.
Observe that the normalization factors may be different for different $\alpha$. However, by \eqref{Eq:volPd} we know that
\begin{equation}\label{key}
C_{\alpha} \leq C_{d,N}N^{-\frac{2}{d-1}} \leq CN^{-\frac{2}{d-1}}.
\end{equation}
Assume that $N \geq C^d$, where $C$ is defined in the previous equation. First, we analyze the term of the second requirement in Assouad's  Lemma (Lemma~\ref{Lem:Assoud}). Indeed, for when $ \|\alpha - \beta\|_0 =1 $, the distributions differ by a one cap, therefore  $|C_\alpha- C_\beta| \leq CN^{-(1+\frac{2}{d-1})}$. We denote by $\mathrm{supp}$ the support of a density and assume without loss of generality that $\alpha_1 = 1$ and $\beta_1 = 0$. Then
\begin{equation}\label{Eq:onecapp}
\begin{aligned}
h^{2}(f_{\alpha},f_{\beta}) & =\int_{B_d}\left(\sqrt{f_{\alpha}(x)}-\sqrt{f_{\beta}(x)}\right)^{2}dx 
\\& \leq \vol(B_d)^{-1}\int_{\text{supp}f_{\alpha} \cap \text{supp}f_{\beta}}((1-C_{\alpha})^{-0.5} -(1-C_{\beta})^{-0.5})^2dx \\& \quad \quad+   \int_{C(x_1,t_{d,N})}\left((1-CN^{-\frac{2}{d-1}})\vol(B_{d})\right)^{-1}dx
\\&  \leq C\vol(B_d)^{-1}\int_{B_d}|C_{\alpha}-C_{\beta}|^2dx +  C_1\vol(B_{d})^{-1}\vol(C(x,t_{n,N}))
\\& \leq CN^{-(2+\frac{4}{d-1})} +  C_1\vol(B_{d})^{-1}\vol(C(x,t_{n,N}))
\\& \leq CN^{-(1+\frac{2}{d-1})} \leq 0.5n^{-1},
\end{aligned}
\end{equation}
where we used \eqref{key} and \eqref{Eq:volPd}, and the fact that for $t$ small enough $(1-t)^{-0.5} = 1 + 0.5t + O(t^2)$.

 Now we prove the first requirement: let $\alpha, \beta \in \{0,1\}^K$, then
\begin{equation}\label{Eq:dimHellinger}
\begin{aligned}h^{2}(f_{\alpha},f_{\beta}) & \geq\int_{\text{supp}f_{\alpha}\triangle \text{supp}f_{\beta}}\left(\sqrt{f_{\alpha}(x)}-\sqrt{f_{\beta}(x)}\right)^{2}dx\\&
	\geq\vol(B_{d})^{-1}\vol (\text{supp}f_{\alpha} \triangle \text{supp}f_{\beta})
\\
&\geq \vol(B_{d})^{-1}\sum_{i=1}^{K}\mathbbm{1}_{\alpha_{i}\ne \beta_{i}}\vol\left(C(x_{i},t_{d,N}) \setminus \bigcup_{1\leq j \leq K,j \ne i}C(x_j,t_{d,N})\right)\\
& \geq c_{2}\|\alpha-\beta\|_{0}N^{-\left(1+\frac{2}{d-1}\right)},
\end{aligned}
\end{equation}
where we used \eqref{Eq:Caps}. 
Using the last two equations, we may set $\eta =c_3 N^{-\left(1+\frac{2}{d-1}\right)}$.
In order for $ CN^{-\left(1+\frac{2}{d-1}\right)} \leq 0.5n^{-1} $ to hold, we set $  K = cN=  c_1n^{\frac{d-1}{d+1}}$.
Then, we apply Assouad's  Lemma with $K,\eta$ and get  
\begin{equation*}
\inf_{\bar{f}}\Rcal_{h^2}(\bar{f},\Fcal_d) \geq c_4 \cdot K \cdot \eta \geq c_2 \cdot cn^{-1} \cdot c_2n^{\frac{d-1}{d+1}} = c_3n^{-\frac{2}{d+1}},
\end{equation*}  
and the claim follows. Finally, we prove Lemma \ref{Lem:Kur}.
\paragraph{Proof of Lemma \ref{Lem:Kur}}
	First we use a lemma proven by  \cite{kur2017approximation}. This lemma shows that there is a polytope with $ N \geq  10^d $ facets that gives an optimal approximation, up to an universal constant, to the Euclidean ball, in terms of the symmetric volume difference.
		\begin{lemma}[\cite{kur2017approximation}, Corollary of Lemma~3.1 (substituting $ \gamma = 1 $)]
			Let $N \geq 10^{d}$ and  $t_{d,N} := \sqrt{1-\left(\frac{\mathrm{vol}(\partial B_d)}{N\mathrm{vol}(B_{d-1})}\right)^{\frac{2}{d-1}}}$. There is a set of points $x_1,\ldots, x_{N} \in \partial B_d$ that defines a body 
			\[
			P_{d,N} :=  B_d \setminus \bigcup_{i=1}^{N}C(x_i,t_{d,N}).
			\]
			\noindent{that satisfies the following:}
			\begin{equation}
			\mathrm{vol}(P_{d,N}) = \left(1- c_{d,N}N^{-\frac{2}{d-1}}\right)\mathrm{vol}(B_d)
			\end{equation}
			\noindent{where $c_{d,N} \in (c_1,C_1) $.}
		\end{lemma}
	Clearly, all the caps have the same volume. We estimate the volume of a cap $C(x,t_{d,N})$, by using the following standard result (see for example, Lemma 5.2 by \citet{kur2017approximation}): If $t_{d,N} > 0.5$, then
	\[
	\vol(C(x,t_{d,N})) = \vol(B_{d-1})\frac{\left(1-t_{d,N}^{2}\right)^{\frac{d+1}{2}}}{t_{d,N}\left(d-1\right)}+O\left(\frac{\vol(C(x,t_{d,N}))}{d}\right).
	\]
	Now observe that $1- t_{d,N} = \sqrt{1-\left(\frac{\vol(\partial B_d)}{\vol(B_{d-1})N}\right)^{\frac{2}{d-1}}} \approx 0.5 \cdot N^{-\frac{2}{d-1}}$, where $\approx $ denotes equality up to a ``small" absolute constant. Then,
	\begin{equation}\label{Eq:volofacap}
		\begin{aligned} \vol(C(x,t_{d,N})) &\approx \vol(B_{d-1})\frac{\left(1-t_{d,N}^{2}\right)^{\frac{d-1}{2}}(1-t_{d,N})}{d} \cdot \frac{1+t_{d,N}}{t_{d,N}}
		\\&	\approx 2\vol(B_{d-1})\frac{\left(1-t_{d,N}^{2}\right)^{\frac{d-1}{2}}}{d}\cdot (1-t_{d,N}) 
		= 2N^{-1}(1-t_{d,N})\frac{\vol(\partial B_d)}{d}
		\\& =2 N^{-1}(1-t_{d,N})\vol(B_d) \approx  N^{-(1+\frac{2}{d-1})}\vol(B_d).
		\end{aligned}
	\end{equation}
	where we used the definition of $ t_{d,N}$, and the last identity in Lemma \ref{Lem:EuclideanBallVolume}. 
	
Clearly, we showed that the volume of each cap is at most  $ \tilde{C}_{d,N}N^{-(1+\frac{2}{d-1})}\vol(B_d) $ where $ \tilde{C}_{d,N} \in (c_1,C_1) $. Finally, we prove the second part of the Lemma (i.e. \eqref{Eq:Caps}). For this purpose, we utilize the following Lemma:
	\begin{lemma} \label{lem:cupsets}
		Let $\mathcal{A}= \left\{S_1, \dots, S_N\right\}$ be a collection of measurable sets in $\mathbb{R}^d$, each of the same volume $v$. Additionally, assume that $\vol\left( \bigcup_{i=1}^N S_i \right) = c N v$, for some constant $0 < c < 1$. Then there exists a subset $\mathcal{A}' \subseteq \mathcal{A}$ of size at least $N c/2 $, such that the following holds:
		\[
		\forall S \in \mathcal{A}' \colon \vol\left( S \setminus \bigcup_{S' \in \mathcal{A}',\ S' \ne S} S' \right) \ge v c/2.
		\]
	\end{lemma}
	\begin{proof}
		In order to return a set that satisfies Lemma \ref{Lem:Kur},  we perform the following greedy algorithm: We initialize a list $ \mathcal{L} \gets \{1,\ldots,N\}$: For each $ i \in \{1, \dots, N\}$, in an increasing order, we check if
		\[
		\vol\left(S_i \setminus \bigcup_{S' \in \mathcal{L},\ S' \ne S_i} S'\right) < v c/2.
		\]
		If the answer is positive, we remove $S_i$ from $\mathcal{L}$, and set $\mathcal{L}\gets \mathcal{L} \setminus \{S_i\}$. Otherwise, we do not change $\mathcal{L}$. Then, we continue with $i+1$.
			At the end of this procedure, we are guaranteed to have $\vol(\bigcup_{i \in \mathcal{L}}S_i) \ge N v c / 2$, since each removal decreases the volume of $\bigcup_{i \in \mathcal{L}}S_i$ by at most $vc/2$. Additionally, since element is of volume $v$, there are at least $cN/2$ remaining elements. The proof follows with $\mathcal{A}'=\mathcal{L}$.
	\end{proof}

	We apply Lemma~\ref{lem:cupsets} with $\mathcal{A} = \left\{ C(x_1, t_{d,N}), \dots, C(x_N, t_{d,N}) \right\}$, $v = \widetilde{C}_{d,N}N^{-\left( 1+\frac{2}{d-1} \right)} \vol(B_d)$, and $c = c_{d,N}$. The lemma concludes by returning the set of caps $\mathcal{A}'$.

\section{Proof of Theorem \ref{Thm:Entropy}}\label{sec:entropy_proof}
Thorough this section, we are going to use of the set $D_{f_0,\epsilon}$ that is defined in the following lemma: 
\begin{lemma}\label{Lem:support}
    For any $f_0 \in \Fcal_{d,\tilde{I}}$, there exists a subset $D_{f_0,\epsilon}$ that satisfies the following:
\begin{enumerate}
	\item  $D_{f_0,\epsilon}$ is a convex set that lies in $B_\epsilon := \left(C_1\max\{\log(\epsilon^{-1}),\sqrt{d}\}\right)B_d$.	
	\item  $3\log(\epsilon) \leq \log(f_0(x)) \leq C_1(d) \quad \forall x \in D_{f_0,\epsilon}$.
	\item  $\int_{D_{f_0,\epsilon}}f_0(x)dx \geq 1-\epsilon^2$.
\end{enumerate}
\end{lemma}
\begin{proof}
First, we use a  corollary of \cite{lee2018stochastic} that implies  that for any $t > 0$ the following holds:
\[
	\Pr_{X \sim f_0 \in \Fcal_{d,\tilde{I}}}(|\|X\|_2 - \sqrt{d}| \geq C\sqrt{d}t) \leq e^{-c\sqrt{d}t}.
\]   
Thus, we know that $1-\epsilon^3$ of the mass of these distributions lies in $B_\epsilon:= C_1\max\{\log(\epsilon^{-1}),\sqrt{d}\}\cdot B_d$. Clearly, it is enough to approximate $f_0$ on this ball.  Also, by Lemma \ref{lem:density-decay}, we know that 
$
	\|f_0\|_{\infty} \leq C(d).
$

Using the fact that the super-level sets are convex, we may restrict ourselves to a convex subset of the ball $B_{\epsilon}$, denoted by $D_{f_0,\epsilon}$, such that $f_0(x) \geq \epsilon^{3}$. Clearly, when $\epsilon$ is small enough the following holds:
\[
    \PP_{f_0}(D_{f_0,\epsilon}) \geq 1-\epsilon^3- \vol(B_{\epsilon}) \cdot \epsilon^{3} \geq 1-\epsilon^{3}- C^d\epsilon^{3}\log(\epsilon^{-1})^d \geq 1-\epsilon^2.
\]
where, in the last inequality we assume that $\epsilon$ is small enough. The claim follows.
\end{proof}
\subsection{Covering with respect to total variation, an upper bound}
Define $$K(f_0) := \{ (x,y): 3\log(\epsilon) \leq y \leq f_0(x), x \in D_{f_0,\epsilon}\},$$ 
where $D_{f_0,\epsilon}$ is defined in Lemma \ref{Lem:support}. Observe that it is convex in $\R^{d+1}$ and lies inside the ball $C(d)\log(\epsilon^{-1})B_{d+1}$. By Lemma \ref{lem:bro}, we have $2^{
	O_d((\epsilon/ \log(\epsilon^{-1})^{d+1})^{-d/2})}$ convex sets, such that for every $f_0$ and its $K(f_0)$, there exists a $\tilde{K}(f_0)$, that satisfies the following properties:
\[
	\vol(\tilde{K}(f_0) \triangle K(f_0)) \leq \bar{c}(d)\epsilon \text{ and } d_{H}(\tilde{K}(f_0),K(f_0)) \leq \bar{c}(d)\epsilon \text{ and } K(f_0) \subset \tilde{K}(f_0).
\]
where $\bar{c}(d)$ is a constant that will be defined later.

Since $K(f_0) \subset \tilde{K}(f_0)$, we can construct a concave function supported on $D_{f_0,\epsilon}$, denoted by $\log(\tilde{f}_0)$, such that
\[
	\int_{D_{f_0,\epsilon}}|\log(\tilde{f}_0) - \log({f}_0)| \leq \epsilon \text{ and }   4\log(\epsilon) \leq \log(\tilde{f}_0) \leq 2C(d).
\]
In order to see this, the first property follows because symmetric volume difference is equivalent to the area between the graphs of the functions. The second property follows from the small Hausdorff distance. Now, we may assume without loss of generality that $\tilde{f}_0$ is a log-concave density on $D_{f_0,\epsilon}$.  To see this, by definition 
\begin{align*}
    \int_{D_{f_0,\epsilon}}\tilde{f_0}dx &\leq \int_{D_{f_0,\epsilon}}f_0dx + \int_{D_{f_0,\epsilon}}|f_0 - \tilde{f_0}|dx \\&\leq  1 - \epsilon^2 + C(d)\int_{D_{f_0,\epsilon}}|\log(f_0) - \log(\tilde{f_0})|dx \leq 1+ C_1(d)\bar{c}(d)\epsilon,
\end{align*}
where we used Lemma \ref{Lem:support} and the fact that  if $e^s,e^t \leq C(d)$ then $|e^s -e^{t}| \leq C(d)|s-t|$. Hence, we can re-scale $\tilde{f}_0$ to be a density, and it causes a perturbation of $C_1(d)\bar{c}(d)\epsilon$. Therefore, when $\bar{c}(d)$ is small enough, the perturbation is less than $\epsilon/3$, and does not affect the result.

Finally, by using similar considerations as in the previous equation, we conclude that
\begin{align*}
    d_{TV}(f_0,\tilde{f_0}) &=  0.5\int_{\R^d}|f_0 - \tilde{f_0}|dx \leq 0.5\int_{D_{f_0,\epsilon}}|f_0 - \tilde{f_0}|dx + \epsilon^2 \\&\leq  C(d)\int_{D_{f_0,\epsilon}}| \log(f_0) - \log(\tilde{f_0})|dx + \epsilon^2 \leq 2C(d)\bar{c}(d)\epsilon \leq \epsilon,
\end{align*}
where we assume that $\bar{c}(d)$ is small enough.

We showed that for every $f_0 \in \mathcal{F}_{d,\tilde{I}}$ there is an $\epsilon$-approximation $\tilde{f}_0$ generated directly from the $\epsilon$-net of the convex sets that lie inside $C(d)\log(\epsilon^{-1})B_{d+1}$. Thus,  $\epsilon$-entropy numbers of the set of almost isotropic log-concave distributions, with respect to the total variation distance,  are bounded by $(\epsilon/ \log(\epsilon^{-1})^{d+1})^{-d/2}$.
\subsection{Covering with respect to Hellinger distance, an upper bound}
In view of Lemma \ref{lem:bro}, there are $2^{{O}_d((\epsilon/\log(\epsilon^{-1})^{d/2})^{-(d-1)}))}$ convex sets such that for every $f_0$ there exists an $\epsilon^2$-approximation to  $D_{f_0,\epsilon}$, denoted by  $\tilde{D}_{f_0,\epsilon}$, that satisfies the following properties:
\[
\vol(\tilde{D}_{f_0,\epsilon} \triangle D_{f_0,\epsilon}) \leq \bar{c}(d)\epsilon^2  \text{ and } \tilde{D}_{f_0,\epsilon} \subset D_{f_0,\epsilon},
\]
where $ D_{f_0,\epsilon} $ is defined in Lemma \ref{Lem:support}, and $\bar{c}(d)$ is small constant that will be defined later.

Now, we want to find an $\epsilon$-covering for all the concave functions bounded by $|\log(\epsilon)|$ that are supported on $\tilde{D}_{f_0,\epsilon} \subset B_{\epsilon}$. For this purpose, we use the main result \cite[Theorem 1.5]{gao2017entropy}, which states that there exits an $\epsilon$-net with bracketing, with respect to the $L_2(\mathrm{Unif}(\tilde{D}_{f_0,\epsilon}))$ metric, with cardinality bounded by $2^{O_d((\epsilon/\log(\epsilon^{-1})^{(d+2)/2})^{-(d-1)})}$. Namely, for every $f_0$, there exists a function in this net, denoted by $\log(\tilde{f_0})$, such that
\[
	\sqrt{\int_{\tilde{D}_{f_0,\epsilon}}(\log(f_0)- \log(\tilde{f_0}))^2dx}:=\| \log(f_0)- \log(\tilde{f_0}) \|_{2,\tilde{D}_{f_0,\epsilon}} \leq \bar{c}(d)\epsilon ~\text{ and }   \log(f_0) \leq \log(\tilde{f_0}) \leq 2C(d).\footnote{The upper bound on $\log(\tilde{f_0})$ follows from their construction, or it can be proven by using Lemma \ref{lem:bro}.}
\]
 Using a similar argument as in the previous sub-section, we may also assume that $\tilde{f_0}$ is a density. Thus, we conclude that
\begin{align*}
    	h(f_0,\tilde{f_0}) &\leq \sqrt{\int_{D_{f_0,\epsilon}}(\sqrt{f_0}-\sqrt{\tilde{f_0}})^2dx + \int_{\R^d \setminus D_{f_0,\epsilon}}f_0(x) dx} \\&= \sqrt{\int_{D_{f_0,\epsilon}}(e^{0.5\log(f_0(x))}-e^{0.5\log(\tilde{f_0})})^2dx} + \epsilon
    	\\&\leq \epsilon + C(d)\|
      \log(f_0)- \log(\tilde{f_0}) \|_{2,D_{f_0,\epsilon}}, 
\end{align*}
where we used Lemma \ref{Lem:support}, and the fact that if $ e^s, e^t \leq C(d)$ then $|e^s-e^t| \leq C_3(d)|s-t|$. Finally,
 we use the definition of $\log(\tilde{f_0})$ and bound the last term  
\begin{align*}
    	\| \log(f_0)- \log(\tilde{f_0}) \|_{2,D_{f_0,\epsilon}} &\leq C(d)\sqrt{\vol(D_{f_0,\epsilon} \triangle \tilde{D}_{f_0,\epsilon})} +  \| \log(f_0)- \log(\tilde{f_0}) \|_{2,\tilde{D}_{f_0,\epsilon}} \\&\leq \bar{c}(d)C(d)\epsilon 
\end{align*}
where we used the fact that both these functions are bounded by $C(d)$, and $\bar{c}(d)$ is chosen to be small enough. Using the last two equations, we conclude that $	h(f_0,\tilde{f_0}) \leq 2\epsilon$.

Observe that we generated a family of functions that gives an $2\epsilon$-covering to the almost isotropic log-concave distributions with respect to the Hellinger distance. Our family was generated by choosing $2^{O_d((\epsilon/\log(\epsilon^{-1})^{d/2})^{-(d-1)})}$ convex sets and for each set we choose $2^{{O}_d((\epsilon/\log(\epsilon^{-1})^{(d+2)/2})^{-(d-1)})}$ bounded concave functions. Thus, the $\epsilon$-entropy numbers are bounded by ${O}_d((\epsilon/\log(\epsilon^{-1})^{(d+2)/2})^{-(d-1)})$, and the claim follows. 

\subsection{Lower bounds for the $\epsilon$-entropy numbers}
The lower bound with respect to the total variation is proved in the second part of Theorem \ref{thm:TV} (see Remark \ref{rem:lowerbound}). The lower bound with respect to the Hellinger distance was proved in \cite{kim2016global}; in the second part of Theorem \ref{thm:hellinger}, we improved the dimensional constant from $c^d$ to $c_1$.

\section{Proof of Theorem \ref{thm:TV}, Lower Bound}
\label{sec:TV_lower}

We start with an overview of the proof. The main idea  is to create a family of log-concave distributions 
$$\mathcal{F}_{\{0,1\}^K} :=\{f_{\alpha}: \alpha  \in \{0,1\}^K\},$$ 
where $ K= \Theta_d(n^{\frac{d}{d+4}})$ and its members are perturbations of the Gaussian distribution, and its size is, by definition,  $2^{\Theta_d(n^{\frac{d}{d+4}})}$. The perturbations will be negligible under the Hellinger squared, but not in the total variation metric. Once we construct the distributions, it will follow that the Hellinger squared distance between a couple of distributions $ f_{\alpha},f_{\beta} $ where $\alpha,\beta \in \{0,1\}^K$, is $$O_d\left(\|\alpha-\beta\|_0 \cdot n^{-1}\right),$$
while in total variation it  is
\begin{equation}\label{Eq:IntroVar}
O_d\left(\|\alpha-\beta\|_0 n^{\frac{2}{d+4}} \cdot n^{-1}\right).
\end{equation}
Assouad's  Lemma (Lemma \ref{Lem:Assoud}) implies that the risk with respect to the squared Hellinger is at least $ O_d(n^{-\frac{4}{d+4}})$. By \eqref{Eq:IntroVar}, switching from Hellinger to total variation, we get an extra factor of $ n^{\frac{2}{d+4}}$. Thus,
\[
	\inf_{\bar{f}}\Rcal_{\TV}(\bar{f},\Fcal_d) \geq O_d(n^{-\frac{2}{d+4}}).
\]  
\subsection{Proof}
Denote by  $\gamma(x)$ the Gaussian density of $\R^{d}$, also we may assume that $\delta \in (0,e^{-Cd})$  and define $g_{x_0,\delta}(x)$ as follows, for all $x_0 \in \mathbb{R}^d$:
\[
	g_{x_0,\delta}(x) = \begin{cases} 
		\frac{1}{4}(\delta - \|x-x_0\|_2)^2 & \|x-x_0\|_2 \leq \delta  \\
		0 & \|x-x_0\|_2 > \delta.
	\end{cases}
\]
Using $g_{x,\delta}$, we can construct log-concave functions:
\begin{lemma}\label{Lem:concavefunction}
	Let $x_1,\ldots, x_l \in \R^{d}$, such that $\|x_i - x_j\| > 2\delta$ for all $i \ne j$. Then the function
	\begin{equation}{\label{Eq:Logcon}}
	f(x) = \prod_{i=1}^{l}e^{g_{x_i,\delta}}(x)\gamma(x)
	= \begin{cases} 
	e^{g_{x_i,\delta}}(x)\gamma(x) & \exists i \text{ such that } \|x-x_i\|_2 < \delta  \\
	\gamma(x) & \text{otherwise}
	\end{cases}
	\end{equation}
	is log-concave.
\end{lemma}
\noindent The proof of this lemma is deferred to Section~\ref{subsec:pf-lem-conc} in the Supplementary. For an intuition, note that $\log \gamma(x) + g_{x,\delta}$ is concave due to the fact that the Hessian of $\log\gamma(x)$ equals $-I$, while the largest singular value in the Hessian of $g_{x,\delta}$ is bounded by $1/2$; hence, the sum of Hessians is negative semi-definite.

For any $\|x-x_0\|_2 < \frac{\delta}{2}$, the following holds:
\begin{equation}\label{Eq:Bounds}
e^{g_{x_0,\delta}}(x) \in \left(1 + \frac{1}{16}\delta^2,1+\frac{1}{4}\delta^2\right),
\end{equation}
where we used $ 1+t \leq e^{t} \leq 1+2t,$ when $ t $ is small enough.
We use the following lemma, which is proved in  in the Supplementary.
\begin{lemma}\label{Lem:Ballsinbin}
	Let $ \delta \in (0,e^{-Cd})$, and define $ K := c^d\frac{\vol(B_d(\sqrt{2d}))}{\vol(B_d(\delta))}.$ Then, there is a set of disjoint antipodal balls $ B_d(x_1,\delta),\ldots,B_d(x_K,\delta),B_d(-x_1,\delta),\ldots,B_d(-x_K,\delta) $ that are inside $  B_d(\sqrt{2d}) $.    
\end{lemma}

 We create a finite set of  log-concave distributions $ \mathcal{F}_{\{0,1\}^{K}} := \{ f_\alpha \colon \alpha \in \{0,1\}^K, f_{\alpha} \in \mathcal{F}_d \}$, where 
\begin{equation}\label{Eq:densityForalpha}
	f_{\alpha}(x) := C^{-1}_{d,\delta}\gamma(x)c\cdot\exp\left(\sum_{i=1}^{K} \alpha_i g_{x_i,\delta}(x) + \sum_{i=1}^{K} (1-\alpha_{i})g_{-x_i,\delta}(x)\right),
\end{equation}
where $C_{d,\delta}$ is the normalization factor.
Due to the Gaussian symmetry, i.e. $ \gamma(x) = \gamma(-x) $ for all $ x\in \R^d,$ and from the definition of $ f_{\alpha} $, the normalization constants are identical for all $\alpha \in \{0,1\}^{K}$. Most importantly, the normalization factor is close to $1$, as stated below:
\begin{lemma}
It holds that $C_{d,\delta} \in (1,1+e^{-Cd}\delta^2).$ 
\end{lemma}
\begin{proof}
Indeed, we can check the normalization factor for $ \alpha = (1,\ldots,1) $: by \eqref{Eq:Logcon},
\begin{align*}
\int_{\R^{d}}f_\alpha(x) dx - 1 &=\sum_{i=1}^{K}\int_{B_{d}(x_{i},\delta)}(e^{g_{x_{i},\delta}}(x)-1)\gamma(x)dx\\
& \leq C\cdot K\cdot\delta^{2}\int_{B_{d}(\delta)}\gamma(x)dx\\
& \leq C\cdot K\cdot\delta^{2}\vol(B_{d}(\delta))\max_{x\in\R^{d}}\gamma(x)\\
& \leq C\cdot\vol\left(B_{d}(\sqrt{2d})\right)\delta^{2}.
\end{align*}
 where we used Lemma \ref{Lem:EuclideanBallVolume} and \eqref{Eq:Bounds}. Thus, the normalization factor $C_{d,\delta} \le 1+ e^{-cd}$, and $C_{d,\delta} \ge 1$ since $e^{g_\alpha(x)} \ge 1$. Observe that we assumed that $\delta < e^{-2Cd}$, and therefore this normalization factor is small.
\end{proof}
Next, we bound the distances between densities in this class.
\begin{lemma} \label{lem:alpha-beta}
For any $\alpha,\beta \in \{0,1\}^{K}$,
\[
c_1^d\|\alpha-\beta\|_{0}\delta^{d+2}\mathrm{vol}(B_{d})
\le d_{\TV}(f_{\alpha},f_{\beta})
\le c_2^d\|\alpha-\beta\|_{0}\delta^{d+2}\mathrm{vol}(B_{d}),
\]
and
\[
c_3^{d}\|\alpha-\beta\|_0\delta^{d+4}\mathrm{vol}(B_d)
\le h^2(f_{\alpha},f_{\beta}) 
\le c_{4}^{d}\|\alpha-\beta\|_{0}\delta^{d+4}\mathrm{vol}(B_{d}).
\]
In particular,
\begin{equation}\label{Eq:corr}
c_5^d h^2(f_{\alpha},f_{\beta}) \leq  \delta^2d_{\TV}(f_{\alpha},f_{\beta}) \leq c_6^d h^2(f_{\alpha},f_{\beta}).
\end{equation}
\end{lemma}
\begin{remark}{\label{rem:lowerbound}}
Before we prove this Lemma, we show that the lower bound for the $\epsilon$-entropy numbers, with respect to the total variation follows from it. Indeed, if we consider all the distributions with $\|\alpha\|_0 = K/2$, then we have a family of distributions such that its cardinally is a least $2^{c^{-d}\cdot \delta^{-d}}$, and the elements are $c_1^{d} \cdot \delta^2$-separated. It is easy to see that, when $\delta < e^{-Cd}$ these distributions are almost isotropic, and the lower bound follows if we set $\epsilon = \delta^2$.
\end{remark}

\begin{proof}
Observe that $f_{\alpha}$ and $f_{\beta}$ differ only on the balls corresponding to  where they differ. Clearly, there are $2\|\alpha-\beta\|_0$ such balls. Define $$A_{\alpha,\beta}:= \left\{ \bigcup_{1 \le i \le K \colon \alpha_i\ne \beta_i}\left(B_d(x_i,\delta) \cup B_d(-x_i,\delta)\right)\right\}.$$ 
We bound $ d_{\TV}(f_{\alpha},f_{\beta}) $ as follows: 
\begin{equation*}
\begin{aligned}d_{\TV}(f_{\alpha},f_{\beta}) & =\int_{\R^{d}}|f_{\alpha}(x)-f_{\beta}(x)|dx\\
& \geq C_{d,\delta}^{-1}\int_{A_{\alpha,\beta}}\left(\prod_{i=1}^{K}e^{|\alpha_{i}-\beta_{i}|g_{x_{i},\delta}}(x)-1\right)\gamma(x)dx\\
& \geq C\min_{y\in B_{d}(\sqrt{2d})}\gamma(y)\int_{A_{\alpha,\beta}}\left(\prod_{i=1}^{K}e^{|\alpha_{i}-\beta_{i}|g_{x_{i},\delta}}(x)-1\right)dx\\
& \geq c_{7}^{d} \|\alpha-\beta\|_{0}\int_{B_{d}(\delta/2)}\left(e^{g_{0,\delta}}(x)-1\right)dx\\
& \geq c_{7}^{d} \|\alpha-\beta\|_{0}\int_{B_{d}(\delta/2)}\frac{\delta^2}{16}dx\\
& \geq c_1^d\|\alpha-\beta\|_{0}\delta^{d+2}\mathrm{vol}(B_{d}),\\
\end{aligned}
\end{equation*}
where we used \eqref{Eq:densityForalpha} and the fact that $e^{g_{x_i,\delta}}(x) \geq 1+\frac{1}{16}\delta^2$ on the $ B(x_i,\frac{\delta}{2})$. Similarly, one can upper bound the total variation distance.

Also, the following bound holds for  the squared Hellinger, 
\begin{equation*}
\begin{aligned}h^{2}(f_{\alpha},f_{\beta}) & =C_{d,\delta}^{-1}\int_{A_{\alpha,\beta}}\left(\sqrt{\prod_{i=1}^{K}\left(e^{|\alpha_{i}-\beta_{i}|g_{x_{i},\delta}}(x)\right)}-1\right)^{2}\gamma(x)dx\\
& \leq C\max_{y\in B_{d}(\sqrt{2d})}\gamma(y)\int_{A_{\alpha,\beta}}\left(\sqrt{\prod_{i=1}^{K}\left(e^{|\alpha_{i}-\beta_{i}|g_{x_{i},\delta}}(x)\right)}-1\right)^{2}dx\\
& \leq c_{1}^{d}\|\alpha-\beta\|_{0}\int_{B_d(\delta)}\left(\sqrt{e^{g_{0,\delta}}(x)}-1\right)^{2}dx\\
& \leq  c_4^{d}\delta^{d+4}\mathrm{vol}(B_{d})\|\alpha-\beta\|_{0},
\end{aligned}
\end{equation*}
where we used the identity $ \sqrt{1+t} \leq 1+0.5t $ when $ t>0  $ is small enough and that $ e^{g_{x_i,\delta}}(x) \leq 1+\frac{1}{4}\delta^2.$  Similarly, one can lower bound the Hellinger squared. This concludes the proof.
\end{proof}

We apply Assouad's  Lemma (Lemma \ref{Lem:Assoud}), to bound the squared Hellinger minimax risk of any estimator. The first requirement in this lemma follows by Lemma~\ref{lem:alpha-beta}, for $$\eta = c^{d}\delta^{d+4}\mathrm{vol}(B_d).$$ 
For the second requirement, for any $\alpha \ne \beta$, we know by Lemma~\ref{lem:alpha-beta} that $$ h^2(f_{\alpha},f_{\beta}) \leq c_4^{d}\delta^{d+4}\mathrm{vol}(B_d)\|\alpha-\beta\|_0.$$ Therefore, we set $\delta$ such that
\[
c_4^{d}\delta^{d+4}\mathrm{vol}(B_d) = \frac{0.5}{n},
\] 
or equivalently
\begin{equation}{\label{Eq:delta}}
     \delta = c_4\mathrm{vol}(B_d)^{\frac{1}{d+4}}n^{-\frac{1}{d+4}} = c_5\sqrt{d}n^{-\frac{1}{d+4}},
\end{equation}
where we used Lemma \ref{Lem:EuclideanBallVolume}.
Finally, apply Assouad's  Lemma with the value $K$ defined above, namely $K = c_1^{d}\delta^{-d}\frac{\mathrm{vol}\left(B_d(\sqrt{2d})\right)}{\mathrm{vol}(B_d)}$ and $\eta = c^{d}\delta^{d+4}\mathrm{vol}(B_d)$ and $c = 1/2$. We get that 
\begin{equation}\label{Eq:HellingerBound}
	\inf_{\bar{f}}\Rcal_{h^2}(\bar{f},\Fcal_d) \geq \sqrt{1-(1/2)^2} \cdot K \cdot \eta \geq c_1^d\delta^4 \geq  c_1^dn^{-\frac{4}{d+4}},
\end{equation}
where we used Lemma~\ref{Lem:EuclideanBallVolume}. Now, in order to switch to the total variation distance, we restrict ourselves to a subset of log-concave estimators that output distributions from $ \mathcal{F}_{\{0,1\}^k} $, namely
\[
	\widetilde{\mathcal{F}}_d := \left\{f: \left(\mathbb{R}^d\right)^n \to \mathcal{F}_{\{0,1\}^{K}}\right\}.
\]
By Eqs. (\ref{Eq:corr}),(\ref{Eq:delta}),(\ref{Eq:HellingerBound}), we conclude that
 \begin{equation*}
 \inf_{\tilde{f} \in \widetilde{\mathcal{F}}_d}\Rcal_{\TV}(\tilde{f},\Fcal_d) 
 \geq c_1^d n^{\frac{2}{d+4}}\inf_{\tilde{f} \in \widetilde{\mathcal{F}}_d}\Rcal_{h^2}(\bar{f},\Fcal_d)  
 \geq c_2^d n^{\frac{2}{d+4}}\inf_{\bar{f}}\Rcal_{h^2}(\bar{f},\Fcal_d)  \geq c_3^dn^{-\frac{2}{d+4}}. 
 \end{equation*}
Using the fact that $\inf_{\bar{f}}\Rcal_{\TV}(\bar{f},\Fcal_d) \ge
\inf_{\tilde{f} \in \widetilde{\mathcal{F}}_d}\Rcal_{\TV}(\tilde{f},\Fcal_d)/2$ (see Lemma \ref{Lem:gene} in the Supplementary) the claim follows.
\section{Proof of Theorem \ref{thm:TV}, Upper Bound}\label{sec:TV_upper}
First, we choose a maximal $\epsilon$-packing with respect to the total variation distance. It is also an $\epsilon$-covering, denoted by $\mathcal{N}_{\TV}(\epsilon)$. Then, we define a family of sets
$$\mathcal{A} := \{ A_{ij} \subset \R^d:  1\leq i < j \leq |\mathcal{N}_{\TV}(\epsilon)|\},$$ where $A_{ij}$ is a set such that $|\PP_{f_i}(A_{ij})- \PP_{f_j}(A_{ij})| = d_{\TV}(f_i,f_j) \geq \epsilon$ for $f_i\ne f_j \in  \mathcal{N}_{\TV}(\epsilon) $. 
The estimator is a ``tournament'', defined for  $X_1,\ldots,X_n$ i.i.d. samples generated from $f_0$ as
$$\bar{f_0}(X_1,\ldots,X_n) := \argmin{f \in \mathcal{N}_{\TV}(\epsilon)}\max_{A \in \mathcal{A}}|\PP_{n,f_0}(A) - \PP_{f}(A)|,$$
where $\PP_{n,f_0}$ is the empirical measure of these samples. 

Using a standard result from empirical processes theory, we know that with high probability the following holds:
$$ \max_{A \in \mathcal{A}}|\PP_{n,f_0}(A) - \PP_{f_0}(A)| \leq C\sqrt{\log |\mathcal{A}|/n} = C\sqrt{\log |\mathcal{N}_{\TV}(\epsilon)|/n}.$$ 
Now, we require that $C\sqrt{\log |\mathcal{N}_{\TV}(\epsilon)|/n} \leq \epsilon$, and we denote by $\tilde{f}_0$  the closest member in the net to $f_0$. Observe that for any $d_{\TV}(f_i,\tilde{f}_0) > 4\epsilon$ the following holds:
\begin{align*}
   \max_{A \in \mathcal{A}}|\PP_{n,f_0}(A) - \PP_{f_i}(A)| \geq  \max_{A \in \mathcal{A}}|\PP_{f_0}(A) - \PP_{f_i}(A)| - \epsilon \geq  \max_{A \in \mathcal{A}}|\PP_{\tilde{f}_0}(A) - \PP_{f_i}(A)| - 2\epsilon > 2\epsilon 
\end{align*}
where  we used the last two equations. However,
\begin{align*}
   \argmin{f \in \mathcal{N}_{\TV}(\epsilon)}\max_{A \in \mathcal{A}}|\PP_{n,f_0}(A) - \PP_{f}(A)| &\leq \max_{A \in \mathcal{A}}|\PP_{\tilde{f}_0}(A) - \PP_{f_0,n}(A)| \\&\leq  \max_{A \in \mathcal{A}}|\PP_{f_0}(A) - \PP_{\tilde{f}_0}(A)| + \epsilon \leq 2\epsilon.
\end{align*}
   
Thus, we showed that $\tilde{f}_0$ is at least a $4\epsilon$-approximation to $f_0$, with respect to the total variation distance, if $ C\sqrt{\log |\mathcal{N}_{\TV}(\epsilon)|/n} \leq \epsilon$. Hence, by the upper bound on the $\epsilon$-entropy in Theorem \ref{Thm:Entropy}, we derive that
\[
	 O_d(\sqrt{(\epsilon/\log(\epsilon^{-1})^{(d+1)})^{-d/2}/n}) \leq \epsilon \Rightarrow \epsilon \approx C_dn^{-\frac{2}{d+4}}\log^{\frac{d(d+1)}{d+4}}(n).
\]
Thus, this estimator gives an $O_d(n^{-\frac{2}{d+4}}\log^{\frac{d(d+1)}{d+4}}(n)) $ approximation, and the claim follows for almost isotropic log-concave distributions.  


Finally, we extend this estimator to an arbitrary log-concave distribution, by using a technique of data standardization (see sub-section \ref{Extsub} in the Supplementary).

\newpage
\title{Supplementary}
\section{Proof outline of the log-concave reduction to convex sets} \label{sec:log-concave-pr}
The proof idea is as follows: let $\hat{f}$ be the maximum likelihood estimator. Hence, by definition of the MLE, $\PP_n(\log \hat{f}) \ge \PP_n(\log f_0)$, or equivalently, $\PP_n \log (f_0/\hat{f}) \le 0$. If we could argue that $\PP\log(f_0/\hat{f}) \le \PP_n \log(f_0/\hat{f}) + O_d(n^{-2/(d+1)})$, this would imply the following:
\begin{equation} \label{eq:8}
h^2(f_0,\hat{f}) \le d_{KL}(f_0\|\hat{f}) = \PP \log(f_0/\hat{f}) \le \PP_n \log(f_0/\hat{f}) + O_d(n^{-2/(d+1)}\log(n)) \le O_d(n^{-2/(d+1)}\log(n)),
\end{equation}
where $d_{KL}$ denotes the KL divergence which is known to upper bound the squared Hellinger distance. Hence, the claim follows. 
The problem here is that the above KL divergence is infinite. Recall that  $\hat{f}$ is supported on the convex hull of the samples. Hence, for any $x$ which is outside the convex hull, $\log(f_0(x)/\hat{f}(x)) = \infty$, which implies that $d_{KL}(f_0\|\hat{f}) = \infty$. 

In order to overcome this problem, one can approximate $\hat{f}$ with a density $f'$ which is approximately $\hat{f}$ (up to negligible factor), but it is non-zero everywhere. 

Therefore, by the triangle inequality it is enough to bound $h^2(f_0,f')$. Clearly, we may assume that that in \eqref{eq:8},  $\hat{f}$ will be replaced by $f'$. Hence, by \eqref{eq:8} it enough to show that $$\PP \log(f_0 / f') \le  (\PP - \PP_n)(\log f_0) +(\PP-\PP_n)(\log(1/f')) + O_d(n^{-2/(d+1)}\log(n)).$$

For this purpose we will bound the two terms. Since $\PP$ is log concave, the first term can be easily bound using standard concentration inequalities.

for the second term, we can assume that $f'$ and $f_0$ are almost isotropic and thus their densities are bounded from above by $C(d)$, hence $\log(1/f') \geq -C_d$. Moreover, up to negligible perturbation we may assume that $ f' \geq O_d(n^{-C(d)})$. Hence, the following holds:
\begin{equation} \label{eq:11}
\begin{aligned}
(\PP - \PP_n) (\log (1/f')) 
&= \int_{t = -C_d}^\infty (\PP - \PP_n) (\{x \colon \log(1/f'(x)) \ge t \})
= \int_{t = -C_d}^\infty (\PP - \PP_n) (\mathbb{R}^d \setminus A_t)\\
&= \int_{t = -C_d}^{C_d \log n} (\PP_n - \PP) (A_t)
+ o(n^{-\frac{2}{d+1}}\log(n))\\
&\le \int_{t = -C_d}^{C_d \log n} C'_d n^{-2/(d+1)}
+ o(n^{-\frac{2}{d+1}}\log(n))
\le O_d(n^{-2/(d+1)} \log n).
\end{aligned}
\end{equation}
where we used the tail formula in the first equality, in the second inequality we used the fact that drawing a point with density that is small is negligible since $f_0$ is log-concave. In the last inequality we used the fact that $A_t$ is a convex set, and applied Corollary~\ref{cor:conv-mcdiarmid}. Thus, the claim follows.   

\section{Remaining Supporting Results}\label{looseends}

The following Lemma can be found for example in  \cite{artstein2015asymptotic}.
\begin{lemma}\label{Lem:packing}
	Let $\epsilon \in (0,c)$ for some universal constant $c$, and let $l_2$ be the Euclidean distance. Then,
	\[
	c_2^{d}\epsilon^{-d} \leq  M_{l_2}(\epsilon,B_d) \leq c_1^{d}\epsilon^{-d},
	\]
	where $M_{l_2}(\epsilon,B_d) $ to be the maximal $\epsilon$-separated on the Euclidean ball with respect to the $l_2$ distance. 
\end{lemma}



\subsection{Proof of Lemma~\ref{Lem:concavefunction}} \label{subsec:pf-lem-conc}
We start with some equivalent definition of a concave function over $\mathbb{R}^d$. First, it follows from definition that a function is concave if and only if it is concave on any affine line, namely: 
\begin{claim}
	A function $\phi \colon \mathbb{R}^d \to \mathbb{R}$ is concave if for any $x_0 \in \mathbb{R}^d$ and any $v \in \mathbb{S}^{d-1}$, the function $\psi \colon \mathbb{R}\ \to \mathbb{R}$ defined by $\psi(t) = \phi(x_0 + t v)$ is concave.
\end{claim}
Clearly a univariate function is concave if and only if for any $x\in \mathbb{R}$ there exists a neighborhood where it is concave. Formally:
\begin{claim}
	A function $\psi \colon \mathbb{R} \to \mathbb{R}$ is concave if an only if for any $x \in \mathbb{R}$ there exists $\epsilon_x > 0$ such that $\psi$ is concave on $(x - \epsilon_x, x+\epsilon_x)$.
\end{claim}
Combining this two claims, we obtain the following equivalence:
\begin{claim} \label{cla:conc-equivalence}
	A function $\phi \colon \mathbb{R}^d \to \mathbb{R}$ is concave if and only if for any $x_0 \in \mathbb{R}^d$ and any $v \in \mathbb{S}^{d-1}$ there exists $\epsilon > 0$ such that the function $\psi_{x_0,v,\epsilon} \colon (- \epsilon, \epsilon) \to \mathbb{R}$ defined by $\psi_{x_0,v,\epsilon}(t) = \phi(x_0 + t v)_{t\in (-\epsilon,\epsilon)}$ is concave.
\end{claim}
Using Claim~\ref{cla:conc-equivalence}, we will prove that $\log f$ is concave. Pick $x \in \mathbb{R}^d$, and divide into cases:
\begin{itemize}
	\item If $\| x - x_i \| > \delta$ for all $1 \le i \le \ell$, then $\log f = \log \gamma(x)$ in a neighborhood of $f$, where $\log \gamma(x)$ is a concave function. In particular, $\psi_{x,v,\epsilon}$ is concave for any $v \in \mathbb{S}^{d-1}$ and a sufficiently small $\epsilon > 0$.
	\item If $\| x - x_i \| < \delta$ for some $1 \le i \le l$, then $\log f(x)$ equals $\log \gamma(x) + \log g_{x_i,\delta}(x)$ in some neighborhood of $x$. Thus,
	\begin{equation*}
	\log f(x)
	= \log \gamma(x) + \log g_{x_i,\delta}(x)
	= \frac{d}{2}\log (2\pi) + \frac{1}{4}(\delta - \|x-x_i\|)^2 - \| x\|^2 / 2,
	\end{equation*}
	which is a concave function as required.
	\item If $\| x - x_i \| = \delta$, then for any $v \in \mathbb{S}^{d-1}$, there exists $\epsilon > 0$ such that $\psi_{x,v,\epsilon}(t) := \log(f)(x + tv)$ has non-increasing derivative, hence it is concave. The key observation is that $\psi_{x,v,\epsilon}(y)$ is differentiable at $y = 0$, while the monotonicity of the derivative for $y \in (-\epsilon,\epsilon) \setminus \{0\}$ follows from the fact that $\psi_{x,v,\epsilon}$ is concave and differentiable there.
\end{itemize}
\qed
\paragraph{Proof of Lemma~\ref{Lem:Ballsinbin}} \label{sec:pr:ballsin}
Using a standard packing argument (Lemma \ref{Lem:packing}) we can find a set of disjoint balls with radius $ \delta $ in $ B_d(\sqrt{2d}) $ of size
\begin{equation}\label{Eq:TwoCaps}
c^d\frac{\vol(B_d(\sqrt{2d}))}{\vol(B_d(\delta))}.
\end{equation}
Now we take two antipodal caps with height $ 1 - 2\delta $, namely,  
$$ S^{1, \delta}:=  C(x_0,1-2\delta) ~; \quad S^{-1,\delta} := C(-x_0,1-2\delta).$$
Since $ \delta < e^{-Cd},$ by standard volume considerations we know $ \vol(S^{1,\delta} \cup S^{-1,\delta}) > (1-(c'c)^d)\vol(B_d(\sqrt{2d}))$, for a sufficiently small $c'$. Thus by \eqref{Eq:TwoCaps} we must have a cap that contains
\[
(0.5c)^d\frac{\vol(B_d(\sqrt{2d}))}{\vol(B_d(\delta))}.
\]
of the disjoint ball, then we return this set and its antipodal, namely we take each ball in this set and also its antipodal. And the claim follows.
\qed
\begin{lemma}{\label{Lem:gene}}
	$\inf_{\bar{f}}\Rcal_{\TV}(\bar{f},\Fcal_d) \ge
	\inf_{\tilde{f} \in \widetilde{\mathcal{F}}_d}\Rcal_{\TV}(\tilde{f},\Fcal_d)/2$.
\end{lemma}
\begin{proof}
	Choose estimator that minimizes the the risk, denoted by $ \bar{f}.$  Then, project it to  $ \widetilde{\mathcal{F}}_d  $ namely, for every $ n $ samples, we define the projection as follows
	\[
	\bar{f}'(x_1,\ldots,x_n) := \arg\min_{f_{\alpha} \in \widetilde{\mathcal{F}}_d}d_{\TV}(\bar{f}(x_1,\ldots,x_n),f_{\alpha}).
	\]
	Now, we use the triangle inequality
	
	\begin{align*}
	\inf_{\tilde{f} \in \widetilde{\mathcal{F}}_d}\Rcal_{\TV}(\tilde{f},\Fcal_d) &\leq \Rcal_{\TV}(\bar{f}',\Fcal_d) = \sup_{f \in \mathcal{F}_d}\E[d_{\TV}(\bar{f}'(X_1,\ldots,X_n),f)] \\&\leq  \sup_{f \in \mathcal{F}_d}\E[d_{\TV}(\bar{f}'(X_1,\ldots,X_n),\bar{f}(X_1,\ldots,X_n))] + \sup_{f \in \mathcal{F}_d}\E[d_{\TV}(\bar{f}(X_1,\ldots,X_n),f)] 
	\\&\leq  \sup_{f \in \mathcal{F}_d}\E[d_{\TV}(f,\bar{f}(X_1,\ldots,X_n))] + \Rcal_{d_{\TV}}(\bar{f},\Fcal_d)
	\\&\leq 2\Rcal_{d_{\TV}}(\bar{f},\Fcal_d).
	\end{align*}
	where $ X_1,\ldots, X_n$ are drawn i.i.d. $ f \in \mathcal{F}_d$. Thus, the claim follows.
\end{proof}
\subsection{Extending to an arbitrary log-concave distribution}\label{Extsub}
The following argument is relatively standard, the idea is to apply an affine transformation on the last $n/3$ samples such that they will be almost isotropic (``standardization"). Then, we use the estimator that we constructed in the previous section on these samples. Using the fact that total variation is affine-invariant, we will apply the inverse transformation (reverse the ''standardization") on the distribution that our estimator gave us. In this way, we generalize to an arbitrary log-concave distribution.

Denote by $Y = \Sigma^{0.5}X + b$ where $X$ is an isotropic log-concave vector. Then, we estimate the empirical covariance matrix $\hat{\Sigma}$ by the first $n/3$ samples. Then, we multiply our next $n/3$ samples by $\hat{\Sigma}^{-0.5}$. Clearly, each of these samples  $\hat{\Sigma}^{-0.5}Y_i$  is distributed as follows:
\[
\hat{\Sigma}^{-0.5}Y =  \Sigma^{-(1/4)}S^{-0.5}\Sigma^{-(1/4)}\Sigma^{1/2}X +\hat{\Sigma}^{-0.5}b =  S^{-0.5}X + \hat{\Sigma}^{-0.5}b,
\] 
where $S$ is the empirical covariance matrix of $n/3$ i.i.d. samples from $X$.
Using standard results, see for example \cite{kim2016global}, with overwhelming probability $\|S^{-0.5}-I\| \leq O_d(n^{-0.5})$. Now we know that $ \hat{\Sigma}^{-0.5}Y $ has a  bounded covariance matrix, namely its trace bounded by $2d$. Thus, we use the next $n/3$ samples in order to estimate the mean of $\hat{\Sigma}^{-0.5}Y$ which is $\hat{\Sigma}^{-0.5}b$. Thus, we can find an estimator of the mean that satisfies 
\[
\|\widehat{(\hat{\Sigma}^{-0.5}b)} - (\hat{\Sigma}^{-0.5}b)\|_2 \leq O_d(n^{-0.5}).
\]
Finally, we use the last $n/3$ samples of $Y$ and apply them the affine transformation $\hat{\Sigma}^{-0.5}Y-\widehat{(\hat{\Sigma}^{-0.5}b)}$. Thus, these last $n/3$ samples are i.i.d and almost isotropic. 

Now, we may use our estimator on these samples, and find an $O_d(\log(n)^{\frac{d}{d+4}}n^{-\frac{2}{d+4}})$-approximated distribution to $\hat{\Sigma}^{-0.5}Y -\widehat{(\hat{\Sigma}^{-0.5}b)}$. Now, using the fact that the total variation is invariant to invertible affine transformations. We  reverse this affine transformation on the approximated distribution, and get an $O_d(\log(n)^{\frac{d(d-1)}{d+4}})n^{-\frac{2}{d+4}})$-approximated distribution to the one of $Y$. 

\subsection{An alternative proof of Lemma \ref{cor:Gen}}\label{Sup:cor:Gen}
First, by adding some constant we can assume that the functions $f,g,f+g$ are positive and bounded by $2\Gamma$. Then, decompose $(f-g)^2$ as
\begin{equation}\label{key1}
(f-g)^2 = 2f^2+ 2g^2 -(f+g)^2.
\end{equation}
Clearly, each of the three terms is a convex function. To see this, note that $h^2$ is convex for for every positive convex function $h : \R^d \to \R$.
Indeed, convexity means
\[
h(\lambda x + (1-\lambda)y) \le \lambda h(x) + (1-\lambda) h(y).
\]
Since both sides of the above inequality are positive and from convexity of $y \mapsto y^2$, we obtain that
\[
h^2(\lambda x + (1-\lambda)y) 
\leq (\lambda h(x) + (1-\lambda)h(y))^2 \leq \lambda h^2(x) + (1-\lambda)h^2(y).
\]
Now we use the tail formula: namely for any positive random variable $Z$ and  increasing function $\varphi \colon [0,\infty) \to \mathbb{R}$ such that $\varphi(0) = 0$, the following holds:
\[
\E_Z [\varphi(Z)] = \int_{0}^\infty \frac{d}{dt} \varphi(t) \Pr[Z \ge t] dt.
\]
Applying the above formula on $Z = h(x)$ and $\varphi(Z) = Z^2$, and for the log-concave $\PP$, 
\[
\int_{\R^d} h^2(x) \PP(dx)= \int_{0}^{\infty}2t\PP(\{x \in \R^d: h(x) \geq t\})dt.
\]
Moreover, if we further assume that $h$ is also bounded by $2\Gamma$, 
\begin{equation}
\begin{aligned}
\sup_{h\in \Hcal_{d,\Gamma}}|\PP_{n}(h^2) -\PP(h^2)| &= \sup_{h\in \Hcal_{d,\Gamma}}|\int_{0}^{2\Gamma}2t(\PP(h \geq t) - \PP_n(h \geq t)) dt| \\
&\le\sup_{h\in \Hcal_{d,\Gamma}}\int_{0}^{2\Gamma}2t|(\PP(h \geq t) - \PP_n(h \geq t))| dt \\
&\le 2 \Gamma \sup_{h\in \Hcal_{d,\Gamma}}\int_{0}^{2\Gamma}|(\PP(h \geq t) - \PP_n(h \geq t))| dt \\
&= 2 \Gamma \sup_{h\in \Hcal_{d,\Gamma}}\int_{0}^{2\Gamma}|( \PP(h < t) - \PP_n(h < t))| dt \\
&\leq 4 \Gamma^2\sup_{C \in \mathcal{K}_d} \left| \PP_n(C) - \PP(C)\right|,
\end{aligned}
\label{key2}
\end{equation}
where the last inequality follows from the fact that the set $\{x \in \mathbb{R}^d \colon h(x) < t \}$ is convex.
Then, by taking expectation on both sides and Corollary \ref{cor:conv-mcdiarmid} we get
\[
\E \sup_{h\in \Hcal_{d,\Gamma}}|\PP_{n}(h^2) -\PP(h^2)| \leq O_d(\Gamma^2 \cdot n^{-\frac{2}{d+1}}).
\]
Applying the last equation on $h = f,g,f+g$ and recalling that $(f-g)^2 = 2f^2 + 2g^2 - (f+g)^2$, the proof follows.

\bibliographystyle{plainnat}
\bibliography{bib}

\begin{thebibliography}{54}
\providecommand{\natexlab}[1]{#1}
\providecommand{\url}[1]{\texttt{#1}}
\expandafter\ifx\csname urlstyle\endcsname\relax
  \providecommand{\doi}[1]{doi: #1}\else
  \providecommand{\doi}{doi: \begingroup \urlstyle{rm}\Url}\fi

\bibitem[Affentranger(1991)]{affentranger1991convex}
F~Affentranger.
\newblock The convex hull of random points with spherically symmetric
  distributions.
\newblock \emph{Rend. Sem. Mat. Univ. Politec. Torino}, 49\penalty0
  (3):\penalty0 359--383, 1991.

\bibitem[An(1997)]{an1997log}
Mark~Yuying An.
\newblock Log-concave probability distributions: Theory and statistical
  testing.
\newblock \emph{Duke University Dept of Economics Working Paper}, \penalty0
  (95-03), 1997.

\bibitem[Artstein-Avidan et~al.(2015)Artstein-Avidan, Giannopoulos, and
  Milman]{artstein2015asymptotic}
Shiri Artstein-Avidan, Apostolos Giannopoulos, and Vitali~D Milman.
\newblock \emph{Asymptotic geometric analysis, Part I}, volume 202.
\newblock American Mathematical Soc., 2015.

\bibitem[Axelrod et~al.(2019)Axelrod, Diakonikolas, Sidiropoulos, Stewart, and
  Valiant]{axelrod2019polynomial}
Brian Axelrod, Ilias Diakonikolas, Anastasios Sidiropoulos, Alistair Stewart,
  and Gregory Valiant.
\newblock A polynomial time algorithm for log-concave maximum likelihood via
  locally exponential families.
\newblock \emph{arXiv preprint arXiv:1907.08306}, 2019.

\bibitem[Bagnoli and Bergstrom(2005)]{bagnoli2005log}
Mark Bagnoli and Ted Bergstrom.
\newblock Log-concave probability and its applications.
\newblock \emph{Economic theory}, 26\penalty0 (2):\penalty0 445--469, 2005.

\bibitem[Balcan and Long(2013)]{balcan2013active}
Maria-Florina Balcan and Phil Long.
\newblock Active and passive learning of linear separators under log-concave
  distributions.
\newblock In \emph{Conference on Learning Theory}, pages 288--316, 2013.

\bibitem[B{\'a}r{\'a}ny(1992)]{barany1992random}
Imre B{\'a}r{\'a}ny.
\newblock Random polytopes in smooth convex bodies.
\newblock \emph{Mathematika}, 39\penalty0 (1):\penalty0 81--92, 1992.

\bibitem[Bellec(2018)]{bellec2018sharp}
Pierre~C Bellec.
\newblock Sharp oracle inequalities for least squares estimators in shape
  restricted regression.
\newblock \emph{The Annals of Statistics}, 46\penalty0 (2):\penalty0 745--780,
  2018.

\bibitem[Birg{\'e} and Massart(1993)]{birge1993rates}
Lucien Birg{\'e} and Pascal Massart.
\newblock Rates of convergence for minimum contrast estimators.
\newblock \emph{Probability Theory and Related Fields}, 97\penalty0
  (1-2):\penalty0 113--150, 1993.

\bibitem[Boucheron et~al.(2013)Boucheron, Lugosi, and
  Massart]{boucheron2013concentration}
St{\'e}phane Boucheron, G{\'a}bor Lugosi, and Pascal Massart.
\newblock \emph{Concentration inequalities: A nonasymptotic theory of
  independence}.
\newblock Oxford university press, 2013.

\bibitem[Bousquet(2002)]{bousquet2002concentration}
Olivier Bousquet.
\newblock Concentration inequalities and empirical processes theory applied to
  the analysis of learning algorithms.
\newblock 2002.

\bibitem[Brazitikos et~al.(2014)Brazitikos, Giannopoulos, Valettas, and
  Vritsiou]{brazitikos2014geometry}
Silouanos Brazitikos, Apostolos Giannopoulos, Petros Valettas, and
  Beatrice-Helen Vritsiou.
\newblock \emph{Geometry of isotropic convex bodies}, volume 196.
\newblock American Mathematical Soc., 2014.

\bibitem[Bronshtein(1976)]{bronshtein1976varepsilon}
EM~Bronshtein.
\newblock $\varepsilon$-entropy of convex sets and functions.
\newblock \emph{Siberian Mathematical Journal}, 17\penalty0 (3):\penalty0
  393--398, 1976.

\bibitem[Brunel(2013)]{brunel2013adaptive}
Victor-Emmanuel Brunel.
\newblock Adaptive estimation of convex polytopes and convex sets from noisy
  data.
\newblock \emph{Electronic Journal of Statistics}, 7:\penalty0 1301--1327,
  2013.

\bibitem[Brunel(2016)]{brunel2016adaptive}
Victor-Emmanuel Brunel.
\newblock Adaptive estimation of convex and polytopal density support.
\newblock \emph{Probability Theory and Related Fields}, 164\penalty0
  (1-2):\penalty0 1--16, 2016.

\bibitem[Carpenter et~al.(2018)Carpenter, Diakonikolas, Sidiropoulos, and
  Stewart]{carpenter2018near}
Timothy Carpenter, Ilias Diakonikolas, Anastasios Sidiropoulos, and Alistair
  Stewart.
\newblock Near-optimal sample complexity bounds for maximum likelihood
  estimation of multivariate log-concave densities.
\newblock In \emph{Conference On Learning Theory}, pages 1234--1262, 2018.

\bibitem[Chan et~al.(2013)Chan, Diakonikolas, Servedio, and
  Sun]{chan2013learning}
Siu-On Chan, Ilias Diakonikolas, Rocco~A Servedio, and Xiaorui Sun.
\newblock Learning mixtures of structured distributions over discrete domains.
\newblock In \emph{Proceedings of the twenty-fourth annual ACM-SIAM symposium
  on Discrete algorithms}, pages 1380--1394. Society for Industrial and Applied
  Mathematics, 2013.

\bibitem[Chiu et~al.(2013)Chiu, Stoyan, Kendall, and Mecke]{chiu2013stochastic}
Sung~Nok Chiu, Dietrich Stoyan, Wilfrid~S Kendall, and Joseph Mecke.
\newblock \emph{Stochastic geometry and its applications}.
\newblock John Wiley \& Sons, 2013.

\bibitem[Cule and Samworth(2010)]{cule2010theoretical}
Madeleine Cule and Richard Samworth.
\newblock Theoretical properties of the log-concave maximum likelihood
  estimator of a multidimensional density.
\newblock \emph{Electronic Journal of Statistics}, 4:\penalty0 254--270, 2010.

\bibitem[Devroye and Lugosi(2012)]{devroye2012combinatorial}
Luc Devroye and G{\'a}bor Lugosi.
\newblock \emph{Combinatorial methods in density estimation}.
\newblock Springer Science \& Business Media, 2012.

\bibitem[Diakonikolas et~al.(2016)Diakonikolas, Kane, and
  Stewart]{diakonikolas2016efficient}
Ilias Diakonikolas, Daniel~M Kane, and Alistair Stewart.
\newblock Efficient robust proper learning of log-concave distributions.
\newblock \emph{arXiv preprint arXiv:1606.03077}, 2016.

\bibitem[Diakonikolas et~al.(2017)Diakonikolas, Kane, and
  Stewart]{diakonikolas2016learning}
Ilias Diakonikolas, Daniel~M Kane, and Alistair Stewart.
\newblock Learning multivariate log-concave distributions.
\newblock pages 711--727, 2017.

\bibitem[Doss and Wellner(2016)]{doss2016global}
Charles~R Doss and Jon~A Wellner.
\newblock Global rates of convergence of the mles of log-concave and s-concave
  densities.
\newblock \emph{Annals of statistics}, 44\penalty0 (3):\penalty0 954, 2016.

\bibitem[Dudley(1999)]{dudley1999uniform}
Richard~M Dudley.
\newblock \emph{Uniform central limit theorems}.
\newblock Number~63. Cambridge university press, 1999.

\bibitem[D{\"u}mbgen and Rufibach(2009)]{dumbgen2009maximum}
Lutz D{\"u}mbgen and Kaspar Rufibach.
\newblock Maximum likelihood estimation of a log-concave density and its
  distribution function: Basic properties and uniform consistency.
\newblock \emph{Bernoulli}, 15\penalty0 (1):\penalty0 40--68, 2009.

\bibitem[D{\"u}mbgen et~al.(2011)D{\"u}mbgen, Samworth, and
  Schuhmacher]{dumbgen2011approximation}
Lutz D{\"u}mbgen, Richard Samworth, and Dominic Schuhmacher.
\newblock Approximation by log-concave distributions, with applications to
  regression.
\newblock \emph{The Annals of Statistics}, 39\penalty0 (2):\penalty0 702--730,
  2011.

\bibitem[Feng et~al.(2018)Feng, Guntuboyina, Kim, and
  Samworth]{feng2018adaptation}
Oliver~Y Feng, Adityanand Guntuboyina, Arlene~KH Kim, and Richard~J Samworth.
\newblock Adaptation in multivariate log-concave density estimation.
\newblock \emph{arXiv preprint arXiv:1812.11634}, 2018.

\bibitem[Gao and Wellner(2017)]{gao2017entropy}
Fuchang Gao and Jon~A Wellner.
\newblock Entropy of convex functions on {$\mathbb{R}^d$}.
\newblock \emph{Constructive approximation}, 46\penalty0 (3):\penalty0
  565--592, 2017.

\bibitem[Gardner et~al.(2006)Gardner, Kiderlen, Milanfar,
  et~al.]{gardner2006convergence}
Richard~J Gardner, Markus Kiderlen, Peyman Milanfar, et~al.
\newblock Convergence of algorithms for reconstructing convex bodies and
  directional measures.
\newblock \emph{The Annals of Statistics}, 34\penalty0 (3):\penalty0
  1331--1374, 2006.

\bibitem[Ghosal and Sen(2017)]{ghosal2017univariate}
Promit Ghosal and Bodhisattva Sen.
\newblock On univariate convex regression.
\newblock \emph{Sankhya A: The Indian Journal of Statistics}, 79\penalty0
  (2):\penalty0 215--253, 2017.

\bibitem[Guntuboyina(2012)]{guntuboyina2012optimal}
Adityanand Guntuboyina.
\newblock Optimal rates of convergence for convex set estimation from support
  functions.
\newblock \emph{The Annals of Statistics}, 40\penalty0 (1):\penalty0 385--411,
  2012.

\bibitem[Guntuboyina and Sen(2013)]{guntuboyina2013covering}
Adityanand Guntuboyina and Bodhisattva Sen.
\newblock Covering numbers for convex functions.
\newblock \emph{IEEE Transactions on Information Theory}, 59\penalty0
  (4):\penalty0 1957--1965, 2013.

\bibitem[Han(2019)]{han2019global}
Qiyang Han.
\newblock Global empirical risk minimizers with" shape constraints" are rate
  optimal in general dimensions.
\newblock \emph{arXiv preprint arXiv:1905.12823}, 2019.

\bibitem[Han and Wellner(2016{\natexlab{a}})]{han2016approximation}
Qiyang Han and Jon~A Wellner.
\newblock Approximation and estimation of s-concave densities via r{\'e}nyi
  divergences.
\newblock \emph{Annals of statistics}, 44\penalty0 (3):\penalty0 1332,
  2016{\natexlab{a}}.

\bibitem[Han and Wellner(2016{\natexlab{b}})]{han2016multivariate}
Qiyang Han and Jon~A Wellner.
\newblock Multivariate convex regression: global risk bounds and adaptation.
\newblock \emph{arXiv preprint arXiv:1601.06844}, 2016{\natexlab{b}}.

\bibitem[Han et~al.(2019)Han, Wellner, et~al.]{han2019convergence}
Qiyang Han, Jon~A Wellner, et~al.
\newblock Convergence rates of least squares regression estimators with
  heavy-tailed errors.
\newblock \emph{The Annals of Statistics}, 47\penalty0 (4):\penalty0
  2286--2319, 2019.

\bibitem[Kim and Samworth(2016)]{kim2016global}
Arlene~KH Kim and Richard~J Samworth.
\newblock Global rates of convergence in log-concave density estimation.
\newblock \emph{The Annals of Statistics}, 44\penalty0 (6):\penalty0
  2756--2779, 2016.

\bibitem[Kim et~al.(2018)Kim, Guntuboyina, Samworth, et~al.]{kim2018adaptation}
Arlene~KH Kim, Adityanand Guntuboyina, Richard~J Samworth, et~al.
\newblock Adaptation in log-concave density estimation.
\newblock \emph{The Annals of Statistics}, 46\penalty0 (5):\penalty0
  2279--2306, 2018.

\bibitem[Koltchinskii(2011)]{koltchinskii2011oracle}
Vladimir Koltchinskii.
\newblock \emph{Oracle Inequalities in Empirical Risk Minimization and Sparse
  Recovery Problems: Ecole d’Et{\'e} de Probabilit{\'e}s de Saint-Flour
  XXXVIII-2008}, volume 2033.
\newblock Springer Science \& Business Media, 2011.

\bibitem[Kur(2019)]{kur2017approximation}
Gil Kur.
\newblock Approximation of the euclidean ball by polytopes with a restricted
  number of facets.
\newblock \emph{To Appear in Studia Mathematica, arXiv:1705.00210}, 2019.

\bibitem[Lee and Vempala(2018)]{lee2018stochastic}
Yin~Tat Lee and Santosh~S Vempala.
\newblock Stochastic localization+ stieltjes barrier= tight bound for
  log-sobolev.
\newblock In \emph{Proceedings of the 50th Annual ACM SIGACT Symposium on
  Theory of Computing}, pages 1122--1129. ACM, 2018.

\bibitem[Lov{\'a}sz and Vempala(2007)]{lovasz2007geometry}
L{\'a}szl{\'o} Lov{\'a}sz and Santosh Vempala.
\newblock The geometry of logconcave functions and sampling algorithms.
\newblock \emph{Random Structures \& Algorithms}, 30\penalty0 (3):\penalty0
  307--358, 2007.

\bibitem[Macbeath(1951)]{macbeath1951extremal}
Alexander~Murray Macbeath.
\newblock An extremal property of the hypersphere.
\newblock In \emph{Mathematical Proceedings of the Cambridge Philosophical
  Society}, volume~47, pages 245--247. Cambridge University Press, 1951.

\bibitem[Rademacher and Goyal(2009)]{RademacherG09}
Luis Rademacher and Navin Goyal.
\newblock Learning convex bodies is hard.
\newblock In \emph{The 22nd Conference on Learning Theory, Montreal, Quebec,
  Canada}, 2009.

\bibitem[Rakhlin et~al.(2017)Rakhlin, Sridharan, Tsybakov,
  et~al.]{rakhlin2017empirical}
Alexander Rakhlin, Karthik Sridharan, Alexandre~B Tsybakov, et~al.
\newblock Empirical entropy, minimax regret and minimax risk.
\newblock \emph{Bernoulli}, 23\penalty0 (2):\penalty0 789--824, 2017.

\bibitem[Samworth(2018)]{samworth2018recent}
Richard~J Samworth.
\newblock Recent progress in log-concave density estimation.
\newblock \emph{Statistical Science}, 33\penalty0 (4):\penalty0 493--509, 2018.

\bibitem[Schneider and Weil(2008)]{schneider2008stochastic}
Rolf Schneider and Wolfgang Weil.
\newblock \emph{Stochastic and integral geometry}.
\newblock Springer Science \& Business Media, 2008.

\bibitem[Schuhmacher and D{\"u}mbgen(2010)]{schuhmacher2010consistency}
Dominic Schuhmacher and Lutz D{\"u}mbgen.
\newblock Consistency of multivariate log-concave density estimators.
\newblock \emph{Statistics \& probability letters}, 80\penalty0 (5-6):\penalty0
  376--380, 2010.

\bibitem[Stanley(1989)]{stanley1989log}
Richard~P Stanley.
\newblock Log-concave and unimodal sequences in algebra, combinatorics, and
  geometry.
\newblock \emph{Annals of the New York Academy of Sciences}, 576\penalty0
  (1):\penalty0 500--535, 1989.

\bibitem[Tsybakov(2003)]{tysbakovnon}
Alexandre~B. Tsybakov.
\newblock \emph{Introduction to Nonparametric Estimation}.
\newblock Springer, 2003.

\bibitem[van~de Geer(2000)]{van2000empirical}
Sara~A van~de Geer.
\newblock \emph{Empirical Processes in M-estimation}, volume~6.
\newblock Cambridge university press, 2000.

\bibitem[van~der Vaart(2000)]{van2000asymptotic}
Aad~W van~der Vaart.
\newblock \emph{Asymptotic statistics}, volume~3.
\newblock Cambridge university press, 2000.

\bibitem[van~der Vaart and Wellner(1996)]{van1996weak}
Aad~W van~der Vaart and Jon~A Wellner.
\newblock Weak convergence.
\newblock In \emph{Weak convergence and empirical processes}, pages 16--28.
  Springer, 1996.

\bibitem[Yang and Barron(1999)]{yang1999information}
Yuhong Yang and Andrew Barron.
\newblock Information-theoretic determination of minimax rates of convergence.
\newblock \emph{Annals of Statistics}, pages 1564--1599, 1999.

\end{thebibliography}

\end{document}